\documentclass[11pt]{article}
\usepackage{latexsym,amsfonts,amssymb,amsmath,amsthm}
\usepackage{graphicx}
\usepackage{tikz}
\usepackage{ulem}

\newcommand{ \bl}{\color{blue}}

\newcommand{ \bk}{\color{black}}

\begin{document}
\setlength{\baselineskip}{16pt}

\parindent 0.5cm
\evensidemargin 0cm \oddsidemargin 0cm \topmargin 0cm \textheight
22cm \textwidth 15cm \footskip 2cm \headsep 0cm

\newtheorem{theorem}{Theorem}[section]
\newtheorem{lemma}{Lemma}[section]
\newtheorem{proposition}{Proposition}[section]
\newtheorem{definition}{Definition}[section]
\newtheorem{example}{Example}[section]
\newtheorem{corollary}{Corollary}[section]

\newtheorem{remark}{Remark}[section]

\newtheorem{comment}{Comment}[section]

\numberwithin{equation}{section}

\def\p{\partial}
\def\I{\textit}
\def\R{\mathbb R}
\def\C{\mathbb C}
\def\u{\underline}
\def\l{\lambda}
\def\a{\alpha}
\def\O{\Omega}
\def\e{\epsilon}
\def\ls{\lambda^*}
\def\D{\displaystyle}
\def\wyx{ \frac{w(y,t)}{w(x,t)}}
\def\imp{\Rightarrow}
\def\tE{\tilde E}
\def\tX{\tilde X}
\def\tH{\tilde H}
\def\tu{\tilde u}
\def\d{\mathcal D}
\def\aa{\mathcal A}
\def\DH{\mathcal D(\tH)}
\def\bE{\bar E}
\def\bH{\bar H}
\def\M{\mathcal M}
\renewcommand{\labelenumi}{(\arabic{enumi})}

\def\disp{\displaystyle}
\def\undertex#1{$\underline{\hbox{#1}}$}
\def\card{\mathop{\hbox{card}}}
\def\sgn{\mathop{\hbox{sgn}}}
\def\exp{\mathop{\hbox{exp}}}
\def\OFP{(\Omega,{\cal F},\PP)}
\newcommand\JM{Mierczy\'nski}
\newcommand\RR{\ensuremath{\mathbb{R}}}
\newcommand\CC{\ensuremath{\mathbb{C}}}
\newcommand\QQ{\ensuremath{\mathbb{Q}}}
\newcommand\ZZ{\ensuremath{\mathbb{Z}}}
\newcommand\NN{\ensuremath{\mathbb{N}}}
\newcommand\PP{\ensuremath{\mathbb{P}}}
\newcommand\abs[1]{\ensuremath{\lvert#1\rvert}}

\newcommand\normf[1]{\ensuremath{\lVert#1\rVert_{f}}}
\newcommand\normfRb[1]{\ensuremath{\lVert#1\rVert_{f,R_b}}}
\newcommand\normfRbone[1]{\ensuremath{\lVert#1\rVert_{f, R_{b_1}}}}
\newcommand\normfRbtwo[1]{\ensuremath{\lVert#1\rVert_{f,R_{b_2}}}}
\newcommand\normtwo[1]{\ensuremath{\lVert#1\rVert_{2}}}
\newcommand\norminfty[1]{\ensuremath{\lVert#1\rVert_{\infty}}}

\title{Non-local dispersal equations with almost periodic dependence. II. Asymptotic dynamics of Fisher-KPP equations}

\author{Maria Amarakristi Onyido\,\, and\,\,
 Wenxian Shen\\
Department of Mathematics and Statistics\\
Auburn University\\
Auburn University, AL 36849 }

\date{Dedicated to Professor Jibin Li on the occasion of his 80th Birthday}
\maketitle

\noindent {\bf Abstract.}
This series of two papers is devoted to the study of the principal spectral theory of nonlocal dispersal operators with almost periodic dependence and the study of the asymptotic dynamics of nonlinear nonlocal dispersal equations with almost periodic dependence. In the first part of the series, we investigated the principal spectral theory of nonlocal dispersal operators
from two aspects: top Lyapunov exponents and generalized principal eigenvalues. Among others, we provided various characterizations of the top Lyapunov exponents and generalized principal eigenvalues, established the relations between them, and studied the effect of time and space
 variations on them. In this second part of the series, we  study the asymptotic dynamics of nonlinear nonlocal dispersal equations with almost periodic dependence applying the principal spectral theory developed in the first part. In particular, we study the existence, uniqueness, and stability of strictly positive almost periodic solutions of Fisher KPP equations with nonlocal dispersal and almost periodic dependence.  By the properties of the asymptotic dynamics of nonlocal dispersal Fisher-KPP equations, we also establish a new property of the generalized principal eigenvalues of nonlocal dispersal operators in this paper.

\bigskip


\bigskip


\newpage

\section{Introduction}
\setcounter{equation}{0}

This   paper is  devoted to the study of the asymptotic dynamics of the following nonlinear nonlocal dispersal equation,
\begin{equation}
\label{main-nonlinear-eq}
\p_t u=\int_D \kappa(y-x)u(t,y)dy+u f(t,x,u),\quad x\in\bar D,
\end{equation}
where $D\subset \RR^N$ is a bounded domain or $D=\RR^N$, and $\kappa(\cdot)$  and $f(\cdot,\cdot,\cdot)$ satisfy

 \medskip

 \noindent {\bf (H1)} {\it $\kappa(\cdot)\in C^1(\RR^{N},[0,\infty))$, $\kappa(0)>0$,
  $\int_{\RR^N}\kappa(x)dx=1$, and there are $\mu,M>0$ such that $\kappa(x)\le e^{-\mu|x|}$ and $|\nabla \kappa|\le e^{-\mu |x|}$
  for $|x|\ge M$.}

 \medskip

 \noindent {\bf (H2)}  {\it $f(t,x,u)$ is $C^1$ in $u$; $f(t,x,u)$ and $f_u(t,x,u)$
are uniformly continuous and bounded  on $(\RR\times \bar D\times E \bl )$ \bk for any bounded set $E\subset \RR$;   $f(t,x,u)$ is almost periodic in $t$ uniformly with respect to $x\in\bar D$ and $u$ in bounded sets of $\RR$;  $f(t,x,u)$ is also almost periodic in $x$ uniformly with respect to $t\in\mathbb{R}$ and $u$ in bounded sets when $D=\mathbb{R}^N$;    $f(t,x,u)+1<0$ for all $(t,x)\in\RR\times\bar D$ and $u\gg 1$;   and $\underset{t\in\RR, x \in \bar{D}}{\sup}f_u(t,x,u) < 0 $ for each $u\ge 0$.}

\medskip

  Typical examples of the kernel function $\kappa(\cdot)$ satisfying { (H1)} include the probability density function of the normal distribution  $\kappa(x)=\frac{1}{\sqrt{(2\pi)^N}}e^{-\frac{|x|^2}{2}}$ and any $C^1$ convolution kernel function supported on a bounded ball $B(0,r)=\{x\in\RR^N\,|\, |x|<r\}$.
 A prototype of $f(t,x,u)$ satisfying { (H2)} is $f(t,x,u)=a(t,x)-b(t,x)u$, where  $a(t,x)$ and  $b(t,x)$  are  bounded and uniformly continuous in $(t,x)\in\RR\times \bar D$;  are almost periodic in $t$
 uniformly with respect to $x\in\bar D$;  are also almost periodic in $x$ uniformly with respect to $t\in\RR$  when $D=\RR^N$; and $\inf_{t\in\RR,x\in\bar D}b(t,x)>0$.

\smallskip

{Dispersal, the mechanism by which a species expands the distribution of its population, is a
central topic in biology and ecology.   Most continuous models related to
dispersal are based upon reaction-diffusion equations  such as
\begin{equation}
\label{random-dirichlet-eq}
\begin{cases}
u_t=\Delta u+u g(t,x,u),\quad x\in \Omega\cr
u=0,\quad x\in \p \Omega,
\end{cases}
\end{equation}
\begin{equation}
\label{random-neumann-eq}
\begin{cases}
u_t=\Delta u+u g(t,x,u),\quad x\in \Omega\cr
\frac{\p u}{\p n}=0,\quad x\in \p \Omega,
\end{cases}
\end{equation}
where $\Omega$ is a bounded smooth domain,
or
\begin{equation}
\label{random-periodic-eq}
u_t=\Delta u+ ug(t,x,u),\quad x\in\RR^N.
\end{equation}
In such  equations, the dispersal is represented by the Laplacian and  is  governed by random walk. It is referred to as random dispersal and   is  essentially a local behavior describing the movement of cells or organisms
between adjacent spatial locations. }

{In reality, the movements  of some organisms
can occur between non-adjacent spatial locations.  For such a model species, one  can think of trees of which seeds and pollens are disseminated on
a wide range. Reaction-diffusion equations are not proper to model such dispersal.}
 The following nonlocal dispersal  equations are commonly used models to integrate the  long range dispersal for populations  having a long range dispersal strategy (see \cite{Fif, GrHiHuMiVi,  HuMaMiVi, LuPaLe, Tur}, etc):
\begin{equation}
\label{dirichlet-kpp-eq} \p_t u=\int_\Omega
\kappa(y-x)u(t,y)dy-u(t,x)+ug(t,x,u),\quad x\in\bar \Omega,
\end{equation}
\begin{equation}
\label{neumann-kpp-eq}
 \p_t u=\int_\Omega \kappa(y-x)[u(t,y)-u(t,x)]dy+ug(t,x,u),\quad x\in\bar \Omega,
\end{equation}
where $\Omega\subset\RR^N$ is a bounded domain,
and
\begin{equation}
\label{periodic-kpp-eq}
 \p_t u=\int_{\RR^N}
\kappa(y-x)[u(t,y)-u(t,x)]dy+ug(t,x,u),\quad x\in\RR^N.
\end{equation}
In equations \eqref{dirichlet-kpp-eq}, \eqref{neumann-kpp-eq}, and  \eqref{periodic-kpp-eq},
 the dispersal kernel $\kappa(\cdot)$ describes the probability to jump from one
location to another and
 the support of  $\kappa(\cdot)$
 can be thought of as the range of dispersion of the cells.

 Observe   that \eqref{dirichlet-kpp-eq}  can be written as
\begin{equation}
\label{dirichlet-kpp-eq1}
 \p_t u=\int_{\RR^N}
\kappa(y-x)[u(t,y)-u(t,x)]dy+ug(t,x,u),\quad x\in\bar\Omega
\end{equation}
 complemented with the following Dirichlet-type boundary condition
\begin{equation}
\label{dirichlet-condition-eq}
u(t,x)=0,\quad x\in\RR^N\setminus \bar\Omega,
\end{equation}
and
\eqref{neumann-kpp-eq} can be written as
\begin{equation}
\label{neumann-kpp-eq1}
 \p_t u=\int_{\RR^N}
\kappa(y-x)[u(t,y)-u(t,x)]dy+ug(t,x,u),\quad x\in\bar\Omega
\end{equation}
complemented with the following Neumann-type boundary condition
\begin{equation}
\label{neumann-condition-eq}
\int_{\RR^N\setminus \Omega}\kappa(y-x)u(t,y)dy=\int_{\RR^N\setminus\Omega}\kappa(y-x)u(t,x)dy,\quad x\in\bar \Omega.
\end{equation}
The reader is referred to \cite{ShXi2} for the relation between \eqref{dirichlet-kpp-eq1}+\eqref{dirichlet-condition-eq} and the  reaction diffusion  equation \eqref{random-dirichlet-eq} with Dirichlet boundary condition,
and the relation between \eqref{neumann-kpp-eq1}+\eqref{neumann-condition-eq} and the  reaction diffusion equation \eqref{random-neumann-eq} with Neumann boundary condition.

Observe also  that \eqref{dirichlet-kpp-eq} (respectively  \eqref{neumann-kpp-eq}, or  \eqref{periodic-kpp-eq}) can be written as \eqref{main-nonlinear-eq} with $D=\Omega$ and $f(t,x,u)=-1+g(t,x,u)$ (respectively  $D=\Omega$ and $f(t,x,u)=-\int_D \kappa(y-x)dy+g(t,x,u)$, or
$D=\RR^N$ and $f(t,x,u)=-1+g(t,x,u)$).
Hence the theory on the  dynamics of \eqref{main-nonlinear-eq} to be developed in this paper  can be applied to  \eqref{dirichlet-kpp-eq}, \eqref{neumann-kpp-eq}, and \eqref{periodic-kpp-eq}.

{
Considering  a population model, among the fundamental dynamical issues are  asymptotic behavior of  solutions with strictly positive initials,
 propagation phenomena of solutions with compact supported  or front-like initials when the underlying environment is unbounded, and the effects of dispersal strategy and spatial-temporal variations on the population dynamics. These dynamical issues  have been extensively studied for population models described by reaction diffusion equations and are quite well understood in many  cases.
Recently there has also been extensive investigation on these dynamical issues for  nonlocal dispersal population
models (see \cite{BaLi0, BaZh, BaLi,  BeCoVo1, Co1, Co3, CoDaMa, CoDaMa1, DeShZh, GaRo1, KaLoSh, LiSuWa, LiWaZh, LiZh, RaSh,  RaShZh,  ShZh1, ShZh2, ShVo, SuLiLoYa, ZhZh1, ZhZh2}, etc.).  However,  the understanding of these  issues
 for  nonlocal dispersal equations  is much less, and, to our knowledge they  have
been essentially investigated in  specific situations such as time and space periodic media or time independent and space heterogeneous media.}

{
The objective of this paper is to study the asymptotic dynamics of  solutions of \eqref{main-nonlinear-eq} with strictly positive initials, which
would provide some foundation for the study of the  propagation dynamics of positive solutions of \eqref{main-nonlinear-eq} with compact supported or front-like initials (see Remark \ref{general-positive-solu-rk}).
We point out that,  in contrast to the Laplacian, the integral operator in \eqref{main-nonlinear-eq}
is not a local operator.
The mathematical analysis of \eqref{main-nonlinear-eq}
appears to be difficult even though the dispersal is represented by a bounded integral operator. Unlike
the case of reaction-diffusion equations, the forward flow associated with \eqref{main-nonlinear-eq} does not have
a regularizing effect. }

 Note that $u(t,x)\equiv 0$ is a solution of \eqref{main-nonlinear-eq}, which is refereed to as the {\it trivial solution}
of \eqref{main-nonlinear-eq}. If $a(t,x)=f(t,x,0)$, then the following nonlocal linear equation
\begin{equation}
\label{main-linear-eq}
\p_t u=\int_D \kappa(y-x)u(t,y)dy+a(t,x)u,\quad x\in \bar D,
\end{equation}
is the linearization of \eqref{main-nonlinear-eq} at this trivial solution. Hence the principal spectral theory  established for \eqref{main-linear-eq} in {\color{red} \cite{OnSh}}
 has its own interests and also plays an important role in the study of the asymptotic dynamics of \eqref{main-nonlinear-eq}
 in {this paper}.

To state the main results on the asymptotic dynamics of strictly positive solutions of \eqref{main-nonlinear-eq} and some new properties of the generalized principal eigenvalues of   \eqref{main-linear-eq} to be established in this paper, we first recall some of the results
on the principal spectrum of \eqref{main-linear-eq} established in {\color{red} \cite{OnSh}}.

\subsection{Principal spectral theory of \eqref{main-linear-eq}}

In this subsection, we recall some of the results
on the principal spectrum of \eqref{main-linear-eq} established in  \cite{OnSh}.

Let
\begin{equation}
\label{X-space-eq}
X(D)= C_{\rm unif}^b(\bar D)=\{u\in C(\bar D)\,|\, u\,\,\, \text{is uniformly continuous and bounded}\}
\end{equation}
with norm $\|u\|=\sup_{x\in D}|u(x)|$.  If no confusion will occur, we may put
$$
X=X(D),
$$
$$
X^+=\{u\in X\,|\, u(x)\ge 0,\quad x\in \bar D\},
$$
and
$$
X^{++}=\{u\in X^+\,|\, \inf_{x\in\bar D} u(x)>0\}.
$$

Throughout the rest of this subsection, we assume that $a(t,x)$ satisfies  the following (H3).

\medskip

 \noindent {\bf (H3)} {\it  $a(t,x)$ is  bounded and uniformly continuous in $(t,x)\in\RR\times \bar D$, and is almost periodic in $t$
 uniformly with respect to $x\in\bar D$, and is also almost periodic in $x$ uniformly with respect to $t\in\RR$  when $D=\RR^N$.}

 \medskip

\noindent  Sometimes, we may also assume that $a(t,x)$ satisfies the following (H3$)^{'}$.

\medskip

\noindent {\bf (H3$)^{'}$} {\it   $a(t,x)$ is limiting almost periodic in $t$ uniformly  with respect to $x\in \bar D$ and is also limiting almost periodic in $x$ when $D=\RR^N$ (see Definition \ref{almost-periodic-def}(2)).}

\medskip

 For any $s\in\RR$ and $u_0\in X$, let $u(t,x;s,u_0)$ be the unique  solution of
\eqref{main-linear-eq} with $u(s,x;s,u_0)=u_0(x)$ (the existence and uniqueness of solutions of \eqref{main-linear-eq}
with given initial function $u_0\in X$ follow from  the general semigroup theory, see \cite{Paz}).  Let $\Psi(t,s;a, D)$ be the solution operator of \eqref{main-linear-eq} on $X$, that is,
\begin{equation*}
\label{Phi-t-s-eq}
\Psi(t,s;a,D)u_0=u(t,\cdot;s,u_0).
\end{equation*}

\begin{definition}
\label{top-lyapunov-exp-def} Let
\begin{equation*}
\label{lyapunov-exp-eq}
\lambda_{PL}(a,D)=\limsup_{t-s\to\infty}\frac{\ln \|\Psi(t,s;a,D)\|}{t-s}, \quad \lambda_{PL}^{'}(a,D)=\liminf_{t-s\to\infty} \frac{\ln \|\Psi(t,s;a,D)\|}{t-s}.
\end{equation*}
$\lambda_{PL}(a,D)$ and $\lambda_{PL}^{'}(a,D)$ are called the {\rm top Lyapunov exponents} of \eqref{main-linear-eq}.
\end{definition}

Let
\begin{equation}
\label{X-script-space-eq}
\mathcal{X}(D)=C_{\rm unif}^b(\RR\times\bar D):=\{u\in C(\RR\times \bar D\,|\, u\,\, \text{is uniformly continuous and bounded}\}
\end{equation}
with the norm $\|u\|=\sup_{(t,x)\in\RR\times \bar D}|u(t,x)|$. In the absence of possible confusion, we may write
$$
\mathcal{X}=\mathcal{X}(D),
$$
$$
\mathcal{X}^+=\{u\in \mathcal{X}\,|\, u(t,x)\ge 0,\quad t\in\RR,\,\, x\in\bar D\},
$$
and
$$
\mathcal{X}^{++}=\{u\in\mathcal{X}^+\,|\, \inf_{t\in\RR,x\in\bar D} u(t,x)>0\}.
$$
Let $L(a):\mathcal{D}(L(a))\subset \mathcal{X}\to \mathcal{X}$
be defined as follows,
$$
(L(a)u)(t,x)=-\p_t u(t,x)+\int_D
\kappa(y-x)u(t,y)dy+a(t,x)u(t,x).
$$
Let
\begin{align}
\label{Lambda-PE-set}
\Lambda_{PE}(a,D)=\big\{\lambda\in\RR\,|\,  & \exists \, \phi\in \mathcal{X},\,\, \inf_{t\in\RR} \phi(t,x)\ge \not\equiv 0,\,\,  { {\rm for\,\, each}\,\, x\in\bar D,\,\,  \phi(\cdot,x)\in W^{1,1}_{\rm loc}(\mathbb{R})\,\, {\rm and}}\nonumber\\
& (L(a)\phi)(t,x)\ge\lambda \phi(t,x)\,\, {\rm for}\,\, a.e.\, t\in \RR\big\}
\end{align}
and
\begin{align}
\label{Lambda-PE-prime-set}
\Lambda_{PE}^{'}(a,D)=\big\{\lambda\in\RR\,|\, &
\exists \, \phi\in \mathcal{X},\,\, \inf_{t\in\RR,x\in\bar D} \phi(t,x)>0,\,\,{{\rm for\,\, each}\,\, x\in\bar D,\,\, \phi(\cdot,x)\in W^{1,1}_{\rm loc}(\mathbb{R})\quad {and }}\nonumber\\
& (L(a)\phi)(t,x)\le\lambda \phi(t,x)\,\, {\rm for}\,\, a.e.\, t\in \RR\big\}.
\end{align}

{We point out that  the condition $ \phi(\cdot,x)\in W^{1,1}_{\rm loc}(\mathbb{R})$ for each  $x\in\bar D$ on the test function $\phi\in\mathcal{X}$ in  \eqref{Lambda-PE-set} and \eqref{Lambda-PE-prime-set} is  needed for the comparison principle (see the proof of Proposition \ref{comparison-prop}).
 In the definition of the sets $\Lambda_{PE}(a,D)$ and $\Lambda_{PE}^{'}(a,D)$   in \cite{OnSh},   this condition was absent,  which is not because it  was  not needed, but was missed.  }

\begin{definition}
\label{generalized-principal-eigenvalue-def} Define
\begin{equation*}
\label{generalized-principal-eigen-eq}
\lambda_{PE}(a,D)=\sup\{\lambda\,|\, \lambda\in\Lambda_{PE}(a,D)\}
\end{equation*}
and
\begin{equation*}
\label{generalized-principal-eigen-eq1}
\lambda_{PE}^{'}(a,D)=\inf\{\lambda\,|\, \lambda\in\Lambda_{PE}^{'}(a,D)\}.
\end{equation*}
Both $\lambda_{PE}(a,D)$ and $\lambda_{PE}^{'}(a,D)$ are called {\rm generalized principal eigenvalues} of \eqref{main-linear-eq}.
\end{definition}

Let
\begin{equation}
\label{time-average-eq}
\hat a(x)=\lim_{T\to\infty}\frac{1}{T}\int_0^T a(t,x)dt
\end{equation}
(see Proposition \ref{almost-periodic-prop} for the existence of
$\hat a(\cdot)$). Let
\begin{equation}
\label{time-space-average1}
\bar a= \frac{1}{|D|}\int_D \hat a(x)dx
\end{equation}
when $D$ is bounded, and
\begin{equation}
\label{time-space-average2}
\bar a=\lim_{q_1,q_2,\cdots,q_N\to\infty}\frac{1}{q_1q_2\cdots q_N}\int_{0}^{q_N}\cdots\int_0^{q_2}\int_0^{q_1} \hat a(x_1,x_2,\cdots,x_N)dx_1dx_2\cdots dx_N
\end{equation}
when $D=\RR^N$ and $a(t,x)$ is almost periodic in $x$ uniformly with respect to $t\in\RR$ (see Proposition \ref{almost-periodic-prop} for the existence of
 $\bar a$). Note that $\hat a(x)$ is the time average of $a(t,x)$, and $\bar a$ is the space average
of $\hat a(x)$.

 If no confusion occurs, we may put $\lambda_{PL}(a)=\lambda_{PL}(a,D)$,
$\lambda_{PL}^{'}(a)=\lambda_{PL}^{'}(a,D)$, $\lambda_{PE}(a)=\lambda_{PE}(a,D)$, and $\lambda_{PE}^{'}(a)=\lambda_{PE}^{'}(a,D)$.
Among others, we proved the following results in \cite{OnSh}.

\begin{proposition}
\label{spectrum-prop}
Assume (H3).
\begin{itemize}

\item[(1)] (Theorem 1.1 in \cite{OnSh})
$$
\lambda_{PL}^{'}(a)=\lambda_{PL}(a)=\lim_{t-s\to\infty} \frac{\ln\|\Psi(t,s;a)u_0\|}{t-s}=\lim_{t-s\to\infty} \frac{\ln\|\Psi(t,s;a)\|}{t-s}$$
for any $u_0\in X$ with
$\inf_{x\in D}u_0(x)>0$.

\item[(2)] (Theorem 1.2 in \cite{OnSh, OnSh1})
$
\lambda_{PE}(a)\le \lambda_{PE}^{'}(a)= \lambda_{PL}(a).
$
If $a(t,x)$ satisfies (H3$)^{'}$, then
$\lambda_{PE}(a)=\lambda_{PE}^{'}(a).
$

\item[(3)]   (Theorem 1.4 in \cite{OnSh, OnSh1})  $\lambda_{PL}(a)\ge { \lambda_{PL}(\hat a)}\ge \sup_{x\in D}\hat a(x)$.

   \item[(4)] (Theorem 1.3 in \cite{OnSh, OnSh1}) $\lambda_{PE}(a)\ge \sup_{x\in D}\hat a(x)$. If $a(t,x)$ satisfies (H3$)^{'}$, then
$\lambda_{PE}(a) \ge \lambda_{PE}(\hat{a}) \ge \sup_{x\in D}\hat a(x).$

\item[(5)] (Theorem 1.3 in \cite{OnSh,OnSh1}) If $D$ is bounded,  $a(t,x)\equiv a(x)$, and $\kappa(\cdot)$ is symmetric, then
$$\lambda_{PE}(a)\ge \bar a+\frac{1}{|D|}\int_D\int_D \kappa(y-x)dydx,
$$
 where $|D|$ is the Lebesgue measure of $D$.

 \item[(6)] (Theorem 1.3 in \cite{OnSh,OnSh1})  If $D=\RR^N$, $a(t,x)\equiv a(x)$  is almost periodic in $x$, and
 $\kappa(\cdot)$ is symmetric,  then
 $$\lambda_{PE}(a)\ge \bar a+1.
 $$

\item[(7)] (Theorem 1.5(1) in \cite{OnSh, OnSh1}) If $a(t,x)\equiv a(x)$ and satisfies $(H3)^{'}$, then
$$
 \lambda_{PE}(a)=\sup\{\lambda\,|\, \lambda\in \tilde \Lambda_{PE}(a)\}=\inf\{\lambda\,|\,
 \lambda\in\tilde \Lambda_{PE}^{'}(a)\}=\lambda_{PE}^{'}(a),
 $$
 where
\begin{align*}
\tilde\Lambda_{PE}(a)=\{\lambda\in\RR\,|\, \exists \, \phi\in {X},\,\,  \phi(x)\ge \not\equiv  0,\,\, \int_D \kappa(y-x)\phi(y)dy+a(x)\phi(x)\ge\lambda \phi(x)\,\, \forall\, x\in \bar D\}
\end{align*}
and
\begin{align*}
\tilde \Lambda_{PE}^{'}(a)= \{\lambda\in\RR\,|\, \exists \, \phi\in {X},\,\,  \inf_{x\in\bar D}\phi(x)>0,\,\, \int_D \kappa(y-x)\phi(y)dy+a(x)\phi(x)\le \lambda \phi(x)\,\, \forall\, x\in \bar D\}.
\end{align*}

\end{itemize}

\end{proposition}

We conclude this subsection with some remark on the generalized principal  eigenvalues and top Lyapunov exponents  of  \eqref{main-linear-eq}.

\begin{remark}
\label{remark-0}
\begin{itemize}

\item[(1)] It remains open whether $\lambda_{PE}(a)=\lambda_{PE}^{'}(a)$ for any $a(\cdot,\cdot)$ satisfying (H3).

\item[(2)]  The test function $\phi$ in the definition of $\lambda_{PE}^{'}(a)$ and $\lambda_{PE}(a)$ is not required to be almost periodic in $t$.
 The definition of $\lambda_{PL}(a)$, $\lambda_{PL}^{'}(a)$, $\lambda_{PE}(a)$, and $\lambda_{PE}^{'}(a)$ applies to the case where $a(t,x)$ is bounded and uniformly continuous in $t\in\mathbb{R}$ and $x\in\bar D$. For such general $a(t,x)$, by the arguments of  Theorem 1.2 in \cite{OnSh},
we have the following relations between $\lambda_{PL}(a)$, $\lambda_{PL}^{'}(a)$, $\lambda_{PE}(a)$, and $\lambda_{PE}^{'}(a)$,
$$
\lambda_{PE}(a)\le \lambda_{PL}^{'}(a)\le\lambda_{PL}(a)\le\lambda_{PE}^{'}(a).
$$

\item[(3)] {Assume that $a(t,x)$ satisfy (H3) with $D=\RR^N$. By the definition of $\lambda_{PE}^{'}(a,D)$, it is easy to see that
\begin{equation}
\label{monotone-domain-eq1}
\lambda_{PE}^{'}(a,D_1)\le\lambda_{PE}^{'}(a,D_2)
\end{equation}
for any $D_1\subset D_2\subset\RR^N$.  But due to the requirement of the continuity of the test functions in $\Lambda_{PE}(a,D)$, it is not clear whether
$\lambda_{PE}(a,D_1)\le\lambda_{PE}(a,D_2)$ also holds for any $D_1\subset D_2$.  In this paper, we will prove this also holds by applying the criteria for the existence of strictly positive entire solutions of \eqref{main-nonlinear-eq} (see Theorem \ref{pev-thm1}).}
\end{itemize}
\end{remark}

\subsection{Main results}

In this subsection, we state the main results of this paper. Throughout this subsection, we assume (H1) and (H2).

  Observe that  a function $u(t,x)$ satisfying  \eqref{main-nonlinear-eq} need not be continuous in $x$.
 In this paper, unless specified otherwise, when we say that $u(t,x)$ is a {\it  solution of \eqref{main-nonlinear-eq} on an interval $I$}, it means that,  for each $t\in I$,  $u(t,\cdot)\in X$, and
the mapping $I\ni t\mapsto u(t,\cdot)\in X$ is differentiable.  Such a solution $u(t,x)$ is clearly  {differentiable in $t$ and} is continuous in both $t$ and $x$.

 Note that, by general semigroup theory (see \cite{Paz}), for any $s\in\RR$ and $u_0\in X$,   \eqref{main-nonlinear-eq}
has a unique (local) solution $u(t,x;s,u_0)$ with $u(s,x;s,u_0)=u_0(x)$.  Moreover, for any $u_0\in X^+$, $u(t,x;s,u_0)$ exists globally, that is, $u(t,x;s,u_0)$ exists for all $t\ge s$ (see
the comparison principle,  Proposition \ref{comparison-prop} (2)).
A solution $u(t,x)$ of \eqref{main-nonlinear-eq} defined for all $t\in\RR$ is called an {\it entire solution}.
An entire solution  $u(t,x)$ of \eqref{main-nonlinear-eq} is said to be  {\it  positive} if $u(t,x)>0$ for any $(t,x)\in \RR\times \bar D$ and {\it strictly positive} if $\inf_{t\in\RR,x\in\bar D} u(t,x)>0$. A strictly positive entire solution $u(t,x)$ of \eqref{main-nonlinear-eq}  is called an {\it almost periodic solution} if
it is almost periodic in $t$ uniformly with respect to $x\in\bar D$ in the case that $D$ is bounded and is almost periodic in both $t$ and $x$ in the case that $D=\RR^N$.

 In the rest of the paper, $u(t,x;s,u_0)$ always denotes the solution of \eqref{main-nonlinear-eq} with
$u(s,\cdot;s,u_0)=u_0\in X$, unless specified otherwise.
 Among others, we prove

\begin{theorem}
\label{uniqueness-thm}
\begin{itemize}
\item[(a)] (Uniqueness)
There is at most one strictly positive  bounded   entire solution of \eqref{main-nonlinear-eq}.

\item[(b)] (Almost periodicity)  Any strictly positive   bounded entire solution of \eqref{main-nonlinear-eq} is almost periodic.

\item[(c)] (Stability)
 If $u^*(t,x)$ is a strictly positive    bounded almost periodic solution of \eqref{main-nonlinear-eq}, then for any $u_0\in X^{++}$,
$$
\lim_{t\to\infty}
\|u(t,\cdot;t_0,u_0)-u^*(t,\cdot)\|_\infty=0.
$$

\item[(d)] (Frequency module) If $u^*(t,x)$ is a strictly positive   bounded almost periodic solution of \eqref{main-nonlinear-eq}, then
$$
\mathcal{M}(u^*)\subset \mathcal{M}(f),
$$
 where $\mathcal{M}(\cdot)$ denotes the frequency module of an almost periodic function.
\end{itemize}
\end{theorem}

\begin{theorem}
\label{existence-thm}
 Let  $a(t,x)=f(t,x,0)$.
\begin{itemize}

 \item[(a)]  (Existence)  Equation \eqref{main-nonlinear-eq}
has a strictly positive   bounded almost periodic  solution if and only if $\lambda_{PE}(a) > 0$.

\item[(b)]  (Nonexistence)  If $\lambda_{PL}(a) < 0$,  then the trivial solution
 $u\equiv 0$ of \eqref{main-nonlinear-eq} is globally asymptotically stable in the sense that for any $u_0 \in \mathcal{X}^{+}$,  $$\|u(t, \cdot ; 0, u_0) \|_{X} \to 0 \; as \; t \to \infty.$$
\end{itemize}
\end{theorem}

\begin{corollary}
\label{cor-1}
 Let  $a(t,x)=f(t,x,0)$.
\begin{itemize}
\item[(a)]
If $\underset{x \in D}{\sup}\hat a(x)>0$, then  equation \eqref{main-nonlinear-eq}
has a strictly positive almost periodic solution.

\item[(b)] If $\kappa(\cdot)$ is symmetric,  $a(t,x)\equiv a(x)$, and   $\bar a> - \frac{1}{|D|}\int_D\int_D \kappa(y-x)dydx$ when $D$ is bounded and  $\bar a>-1$ when $D=\RR^N$, then equation \eqref{main-nonlinear-eq}
has a strictly positive almost periodic solution.
\end{itemize}
\end{corollary}

\begin{proof} (a)   By \cite[Theorem 1.3(1)]{OnSh}, $\lambda_{PE}(a)\ge \sup_{x\in D}\hat a(x)$;  (a) then follows from  Theorems \ref{uniqueness-thm} and \ref{existence-thm}.

(b)  By \cite[Theorem 1.3(2),(3)]{OnSh},
$\lambda_{PE}(a)\ge \bar a+ \frac{1}{|D|}\int_D\int_D \kappa(y-x)dydx$ when $D$ is bounded and  $\lambda_{PE}(a)\ge \bar a+1$ when $D=\RR^N$; (b) then follows from Theorems \ref{uniqueness-thm} and \ref{existence-thm}.
\end{proof}

We also establish a new property of the generalized principal eigenvalues of \eqref{main-linear-eq} on the domain $D$.

\begin{theorem}
\label{pev-thm1}  {Suppose that $a(t,x)$ satisfies (H3) with $D=\RR^N$.}
For any $D_1\subset D_2$,  there holds
\begin{equation}
\label{monotone-domain-eq2}
\lambda_{PE}(a,D_1)\le \lambda_{PE}(a,D_2),
\end{equation}
{where $D_1$ is bounded and $D_2$ is bounded or $D_2=\RR^N$.}
\end{theorem}

{We point out that in this paper as well as \cite{OnSh}, we consider
$\lambda_{PE}(a,D)$   with $D$ being either bounded or the whole space $\RR^N$. Hence it is assumed that $D_1$ is bounded in Theorem \ref{pev-thm1}. For otherwise, if $D_1=\RR^N$, then $D_2=\RR^N$ and nothing needs to be proved.  }

\subsection{Comments on the main results}

In this subsection, we give some comments on the main results of this paper.

First, we give some comments on our results in some special cases.

\begin{comment} [Extension of existing results in special cases]
\label{special-cases-rk}
$\quad$

 \begin{itemize}

    \item[(1)]   For the case that the function $f(t,x,u)$ is time independent or time periodic  and is periodic in $x$ when $D=\mathbb{R}^N$, similar results on the asymptotic dynamics of \eqref{main-nonlinear-eq} as Theorems \ref{uniqueness-thm} and \ref{existence-thm} have been obtained  in \cite{BaZh, Co1,  RaSh,  ShZh1}).  Our results  recover those  results  in \cite{BaZh, Co1,  RaSh,  ShZh1}.

    \item[(2)]   In \cite{Sh}, results similar to theorems \ref{uniqueness-thm} and \ref{existence-thm} for the case $D=\RR$ were obtained for general time dependence under the condition   $\underset{t-s \to \infty}{\liminf}\frac{1}{t-s}\int_s^t\underset{x\in\RR}{\inf}f(\tau,x,0)d\tau > -1$. Note that when $D=\RR$ and $f(t,x,u)$ is almost periodic in $t$ uniformly with respect to $x\in\RR$, $a_{\inf}(t):=\inf_{x\in\RR} a(t,x)$ is also almost periodic in $t$, where $a(t,x)=f(t,x,0)$. Note also that $\lambda_{PE}(a)\ge \lambda_{PE}(a_{\inf})=1+\lim_{t-s\to\infty}\frac{1}{t-s}\int_s^t a_{\inf}(\tau)d\tau$. Thus our results extend Theorem 2.1 of \cite{Sh} in the case when $D=\RR$ and  $f(t,x,u)$ is almost periodic in $t.$

\item[(3)]   It should be pointed out that, in the case that $D=\RR^N$ and  $f(t,x,u) \equiv f(x,u)$ is not almost periodic in $x$, the existence, uniqueness, and stability  of positive solutions when $\lambda_{P}(a) < 0$ were  established in \cite[Theorem 1.1]{BeCoVo1}, where $a(x)=f(x,0)$ and
$$
\lambda_p(a):=\sup\{\lambda\in\RR\,|\, \exists\, \phi\in C(\RR^N),\, \,\phi>0\,\, s.t.\,\, \int_{\RR^N}\kappa(y-x)\phi(y)dy+a(x)\phi+\lambda\phi\le 0\}.
$$
{Note that the test function in the definiiton $\lambda_p(a)$ may not be uniformly continuous and bounded, which are required
for the test functions in the definition of $\lambda_{PE}^{'}(a)$.  Hence,
 $$\tilde \Lambda_{PE}^{'} (a)\subset -\Lambda_p(a).
$$
In the case that $a(x)$ is limiting almost periodic, by Proposition  \ref{spectrum-prop} (7), we then have
$$
\lambda_{PE}(a)=\lambda_{PE}^{'}(a)\ge -\lambda_p(a).
$$
 Hence, in such a case, $\lambda_p(a) < 0$  implies $\lambda_{PE}(a)>0$ and  our results improve \cite[Theorem 1.1]{BeCoVo1} in the sense that the positive solution we obtained is strictly positive and almost periodic.}
\end{itemize}

\end{comment}

Second,  we give some comments on the time and space variations.

\begin{comment} [Effects of time and space variations]
\label{variation-rk}
$\quad$

\begin{itemize}
\item[(1)]
If $a(t,x)=f(t,x,0)$ is limiting almost periodic, then $\lambda_{PE}(a)\ge \lambda_{PE}(\hat a)$
(see Proposition \ref{spectrum-prop}(4)), which   shows that time variation does not reduce  the generalized principal eigenvalue $\lambda_{PE}$. Thus  Theorem \ref{existence-thm}(b) indicates that time variation may favor the persistence of species.

\item[(2)]  If $a(t,x)=f(t,x,0)$ is independent of $t$, $\kappa(\cdot)$ is symmetric, and $D=\RR^N$, then
$\lambda_{PE}(a)\ge \bar a+1=\lambda_{PE}(\bar a)$ (see Proposition \ref{spectrum-prop}(6)).     Theorem \ref{existence-thm}(b) then  indicates that space variation may favor the persistence of species.
\end{itemize}
\end{comment}

Third, we give some comments on the proofs of the main results.

\begin{comment}[Difficulties in the proofs]
\label{proof-rk}
By Theorem \ref{existence-thm}, $\lambda_{PE}(a)>0$ is
 a necessary and sufficient condition
 for the existence of a unique strictly positive almost periodic solutuon of \eqref{main-nonlinear-eq}, where
$a(t,x)=f(t,x,0)$.  Note that $\lambda_{PE}(a)>0$ indicates that the trivial solution $u=0$ of \eqref{main-nonlinear-eq} is unstable.
It is naturally expected that the instability of the trivial solution $u=0$ implies the existence of a  positive entire solution.
In fact, this has been proved  for the random dispersal counterpart of \eqref{main-nonlinear-eq}.
  However, thanks to  the lack of  the regularizing effect of the forward flow associated with \eqref{main-nonlinear-eq}
and  the lack of Poincar\'e map in non-periodic time dependent case,
it is very nontrivial to prove the existence of strictly positive almost periodic solutions of \eqref{main-nonlinear-eq}.
\end{comment}

Fourth, we  give some comments on  the extension of the main results  to more general cases.

\begin{comment} [Extension of the main results to non-almost periodic cases]
\label{extension-rk}
As mentioned in Remark \ref{remark-0}, the definitions of $\lambda_{PL}(a),\lambda_{PL}^{'}(a),\lambda_{PE}(a),$ and $\lambda_{PE}^{'}(a)$ apply to general $a(t,x)$ which is bounded and uniformly continuous.
 When $f(t,x,u)$ is not assumed to be almost periodic in $t,$ if $\lambda_{PE}(a)>0$ ($a(t,x)=f(t,x,0)$),  we still have a positive continuous function $u^*(t,x)$ which satisfies \eqref{main-nonlinear-eq} for all $t\in\RR$ and $x\in\bar D$.
Moreover, if  $D$ is bounded, then $u^*(t,x)$ is a
 strictly positive entire solution of \eqref{main-nonlinear-eq} and  is asymptotically stable with respect to positive perturbations.
But in general, $u^*(t,x)$ may not be strictly positive (see Remark \ref{positive-entire-solution-rk}).
\end{comment}

Finally, we give some comments on the application of the main results to the study of propagation phenomena in \eqref{main-nonlinear-eq} when $D=\RR^N$.

\begin{comment} [Propagation dynamics]
\label{general-positive-solu-rk}
Suppose that  $D=\RR^N$ and $\lambda_{PE}(a)>0$, where $a(t,x)=f(t,x,0)$.
Then by Theorems \ref{uniqueness-thm} and \ref{existence-thm}, \eqref{main-nonlinear-eq} has
a unique strictly positive almost periodic solution $u^*(t,x)$ that attract all solutions with strictly positive initials uniformly,
but $u^*(t,x)$ does not attract solutions with compactly  supported or front-like initials uniformly.
Biologically, such an  initial indicates that the population initially resides in a bounded region or in one side of the whole space.
Naturally, the population with such initial distribution will  spread into the region where there is no population initially as time evolves. It is interesting to ask how fast the population spreads.
 Based on the investigation in the time independent or periodic case
(see \cite{RaShZh, ShZh1}), it is equivalent to ask how fast the region where the solution is near $u^*(t,x)$ grows.
To be a little more precise,
for a given compact supported initial $u_0$
 (i.e. $u_0(x)\ge 0$ and $\{x\in\RR^N\,|\, u_0(x)>0\}$ is bounded and non-empty) or front-like initial $u_0$ (i.e.
$u_0(x)-u^*(x,0) \to 0$ as $x\cdot\xi\to -\infty$ and $u_0(x)=0$ for $x\cdot\xi\gg 1$ for some unit vector $\xi\in\RR^N$),
and given $0<\epsilon\ll 1$, let
$$
D(t,u_0)=\{x\in\RR^N\, : \, |u(t,x;0,u_0)-u^*(t,x)|\le \epsilon\}.
$$
By the stability of $u^*(t,x)$ with respect to strictly positive initials, it is expected that $D(t,u_0)$ grows as $t$ increases. It is interesting to know  how fast $D(t,u_0)$ grows. We plan to study
this problem somewhere else.
\end{comment}

The rest of the paper is organised as follows. In section 2 we give the preliminary definitions and results to be used in the rest of the paper.
Theorems \ref{uniqueness-thm} and \ref{existence-thm} are proved in Sections 3 and 4 respectively. Section 5 is devoted to the proof of Theorem \ref{pev-thm1}.

\section{Preliminary}

In this section, we present the preliminary definitions and results to be used in the rest of the paper.

\subsection{Almost Periodic functions}

In this subsection, we recall the definition  and present some basic properties of almost periodic functions.

\begin{definition} \label{almost-periodic-def}
\begin{itemize}

\item[(1)]  A bounded function $f\in\C(\RR,\RR)$ is said to be {\rm almost periodic} if for any $\epsilon>0$,
the set
$$
T(f,\epsilon)=\{\tau\in\RR\,|\, |f(t+\tau)-f(t)|<\epsilon\quad \forall\, t\in\RR\}
$$
is relatively dense in $\RR$.

\item[(2)]  Let $f(t,x) \in C(\mathbb{R} \times E, \mathbb{R})$, where $E$ is a subset of $\mathbb{R}^N$, $f(t,x)$ is said to be almost periodic in t uniformly with respect to  $x \in E$, if it is uniformly continuous on
$\RR\times E$ and for any fixed $x \in E$, $f(t,x)$ is an almost periodic function of $t$.

\item[(3)] Let $E\subset\RR^N$ and $f\in C(\RR\times E,\RR)$. $f$ is said to be {\rm limiting almost periodic in $t$ uniformly with respect to $x\in E$}  if there is a sequence $f_n(t,x)$ of uniformly continuous functions which are periodic in $t$ such that $$
    \lim_{n\to\infty} f_n(t,x)=f(t,x)
    $$
    uniformly in $(t,x)\in\RR\times E$.

    \item [(4)] Let $f \in C(\RR\times \RR^N,\RR)$.  $f(t,x)$ is said to be {\rm almost periodic in $x$ uniformly with respect to $t\in\RR$} if $f$ is uniformly continuous in $(t,x)\in\RR\times\RR^N$ and for each $1\le i\le N$, $f(t,x_1,x_2,\cdots,x_N)$ is almost periodic
    in $x_i$.

\item[(5)]  Let $f(t,x) \in  C(\RR \times E, \RR)$ be an almost periodic function in $t$ uniformly with respect to $x\in E\subset \RR^N$. Let $\Lambda$ be the set of real numbers $\lambda$ such that
$$ a(x,\lambda,f):=\lim_{T\to\infty} \int_{0}^{T} f(t,x)e^{-i\lambda t}\; dt $$
 is not identically zero for $x \in E$. The set consisting of all real numbers which are linear combinations of elements of the set $\Lambda$ with integer coefficients is called the  {\rm frequency module of $f(t,x)$},  which we denote by $\mathcal{M}(f).$

\end{itemize}

\end{definition}

\begin{lemma}\label{ptw def}
A function $f(t,x)$ is almost periodic in $t$ uniformly with respect to $x\in E\subset \RR^K$  if and only if
it is uniformly continuous on $\RR\times E$ and
 for every pair of sequences $\{s_n\}_{n = 1}^{\infty}, \; \{r_m\}_{m = 1}^{\infty}$, there are subsequences $\{s_n^{'}\}_{n = 1}^{\infty} \subset \{s_n\}_{n = 1}^{\infty} , \; \{r_m^{'}\}_{m = 1}^{\infty} \subset \{r_m\}_{m = 1}^{\infty}$ such that for each $(t,x) \in\RR \times \RR^K$,
$$ \underset{m \to \infty}{\lim}\underset{n \to \infty}{\lim}f(t + s_n^{'} + r_m^{'},x) = \underset{n \to \infty}{\lim}f(t + s_n^{'} + r_n^{'},x).
$$
\end{lemma}
\begin{proof}
See \cite[Theorems 1.17 and 2.10 ]{fink}.
\end{proof}

\begin{definition}
For an almost periodic function $a(t,x)$ in $t$, the value
$$\hat a(x) := \lim_{t\to\infty} { \frac{1}{t}}\int_{0}^{t} a(t,x) dt. $$
is called the mean value of $a$
\end{definition}

%%%%%%%%%%%%%%%%%%%%%%%%%%%%%%%%%%%%%
\begin{proposition}
\label{almost-periodic-prop}
\begin{itemize}
\item[(1)] If $f(t,x)$ is almost periodic in $t$ uniformly with respect to $x\in E$, then for any sequence $\{t_n\}\subset\RR$, there is a subsequence $\{t_{n_k}\}$ such that the limit $\lim_{k\to\infty} f(t+t_{n_k},x)$ exists uniformly in
    $(t,x)\in\RR\times E$.

\item[(2)] If $f(t,x)$ is almost periodic in $t$ uniformly with respect to $x\in E$, then
 the limit
 $$\hat f(x):=\lim_{T\to\infty}\frac{1}{T}\int_0^T f(t,x)dt
 $$
   exists uniformly with respect to $x\in E$.
 If $E=\RR^N$  and for each $1\le i\le N$, $f(t,x_1,x_2,\cdots,x_N)$ is also almost periodic in $x_i$ uniformly with respect to
 $t\in\RR$  and $x_j\in\RR$ for $1\le j\le N$, $j\not =i$, then  the limit
 $$
 \bar f:=\lim_{q_1,q_2,\cdots,q_N\to \infty}\frac{1}{q_1 q_2\cdots q_N}\int_0^{q_N}\cdots\int_0^{q_2}\int_0^{p_N} \hat f(x_1,x_2,\cdots,x_N)dx_1dx_2\cdots  dx_N$$
  exists.

\end{itemize}
\end{proposition}

\begin{proof}
(1) It follows from \cite[Theorem 2.7]{fink}

(2) It follows from \cite[Theorem 3.1]{fink}

\end{proof}

\begin{proposition}\label{mod cont}
Let $f, \; g \in C(\R \times \R^N, \R)$ be two almost periodic functions in $t$ uniformly with respect to $x$ in bounded sets. $\mathcal{M}(g) \subset \mathcal{M}(f)$ if and only if for any sequence $\{t_n\} \subset \R$, if $\underset{n\to\infty}{\lim}f(t+t_n,x) = f(t,x)$ uniformly for $t\in \R$ and $x$ in bounded sets, then there is $\{t_{n_k}\}$ a subsequence of $\{t_n\}$ such that $\underset{k\to\infty}{\lim}g(t+t_{n_k},x) = g(t,x)$ uniformly for $t\in \R$ and $x$ in bounded sets.
\end{proposition}

\begin{proof}
See \cite[Theorem 4.5]{fink}
\end{proof}
\bk

\subsection{Comparison principle}

In this subsection, we introduce super- and sub-solutions of \eqref{main-nonlinear-eq} in some general sense and present
a comparison principle and some related  properties for  solutions of \eqref{main-nonlinear-eq}.

 Recall that,  for any $s\in\RR$ and $u_0\in X$,  $u(t,x;s,u_0)$ denotes the unique solution of  \eqref{main-nonlinear-eq}  with $u(s,x;s,u_0)=u_0(x)$. Let  $T_{\max}(s,u_0)\in (0,\infty]$ be the largest number such that $u(t,x;s,u_0)$ exists on $[s,s+T_{\max}(s,u_0))$. To indicate the dependence of $u(t,x;s,u_0)$ on $D$,
we may write it as $u(t,x;s,u_0,D)$.

\begin{definition}
A continuous function $u(t,x)$ on $[t_0,t_0+\tau)\times\bar D$ is called a {\rm  super-solution (or sub-solution)} of \eqref{main-nonlinear-eq} on $[t_0,t_0+\tau)$  if for any $x \in \bar D$,  { $u(\cdot,x)\in W^{1,1}(t_0, t_0+\tau)$}, and satisfies,
\begin{equation}
\label{super-sub-solution-eq}
\frac{\partial u}{\partial t}(t,x) \geq (or \leq)
\int_D \kappa(y-x)u(t,y)dy + u(t,x)f(t,x,u)\quad {a.e. \,\, t\in (t_0, t_0+\tau).}
\end{equation}
\end{definition}

{Note  that, in literature, super-solutions  (or sub-solutions) of \eqref{main-nonlinear-eq} on $[t_0, t_0+\tau)$  are defined to be functions $u(\cdot,\cdot)\in C^{1,0}([t_0, t_0+\tau)\times \bar D)$ satisfying \eqref{super-sub-solution-eq} for all $t\in (t_0,t_0+\tau)$ and $x\in\bar D$.  Super-solutions (sub-solutions) of \eqref{main-nonlinear-eq} defined  in the above are more general. Nevertheless, we still have the following
comparison principle.
}

\begin{proposition}
\label{comparison-prop}
(Comparison Principle)
\begin{itemize}
    \item[(1)]  If $u^1(t,x)$ and $u^2(t,x)$ are bounded { sub and super-solutions} of \eqref{main-nonlinear-eq} on $[0, \tau)$ and $u^1(0,\cdot) \leq u^2(0,\cdot)$, then $u^1(t,\cdot) \leq u^2(t,\cdot)$ for $t \in [0,\tau)$.

    \item[(2)]   For every $u_0 \in X^+$, $u(t,x;s,u_0)$ exists for all  $t \geq s.$
\end{itemize}
\end{proposition}

\begin{proof}
(1)
{Set $v(t,x) = e^{ct}(u^2(t,x) - u^1(t,x))$. Then  for each $x\in\bar D$, $v(t,x)$ satisfies
\begin{equation}\label{veqn}
\frac{\partial v}{\partial t} \ge \int_D \kappa(y-x)v(t,y)dy+p(t,x)v(t,x) ~~\text{for}~~
a.e. \;\;t\in [0,\tau),
\end{equation}
where $p(t,x)=a(t,x) + c$,
$$a(t,x)=\int_0^ 1 \frac{\p }{\p s} \Big( (su^2(t,x)+(1-s)u^1(t,x))f(t,x,  su^2(t,x)+(1-s)u^1(t,x))\Big) ds,
$$  and  $c>0$ is  such that $p(t,x) >0$ for all $t \in \RR$ and  $x \in D$. Since $u^i(\cdot,x)\in W^{1,1}(0,\tau)$ for each $x\in\bar D$,  by \cite[Theorem 2, Section 5.9]{Eva}, we have that
\begin{align*}
v(t,x)-v(0,x)&=\int_0^t v_t(s,x)ds\\
&\ge \int_0 ^t   \Big(\int_D \kappa(y-x)v(s,y)dy+p(s,x)v(s,x)\Big)ds
 \quad \forall\, t\in (0,\tau),\,\, x\in\bar D.
\end{align*}
The rest of the proof follows  from the arguments in Proposition 2.1 of \cite{ShZh1}.}

(2) Note that $u\equiv 0$ is an entire solution of \eqref{main-nonlinear-eq} and $u\equiv M$ is a super-solution of \eqref{main-nonlinear-eq} when $M\gg 1$. By (1),
$$
0\le u(t,x;s,u_0)\le M\quad \forall\,\, t\in [s,s+T_{\max}(s,u_0)),\,\, x\in\bar D,\,\, M\gg 1.
$$
This implies that $T_{\max}(s,u_0)=\infty$ and (2) follows.
\end{proof}

\begin{proposition}
\label{new-comparision-prop}
 Let $D_0\subset D$.    Then
$$ u(t,x;s,u_0|_{D_0},D_0)\le u(t,x;s,u_0,D)\quad \forall t\ge s, \,\, x\in \bar D_0,
$$
where $u_0\in C_{\rm unif}^b(\bar D)$, $u_0\ge 0$.
\end{proposition}

\begin{proof}
Observe that $u(t,x;s,u_0,D)$ solves
\begin{eqnarray*}\label{comp eqn}
u_t &=& \int_{D}\kappa(y-x)u(t,y)dy+u(t,x)f(t,x,u),\quad x\in\bar D.\\
&\ge& \int_{D_0}\kappa(y-x)u(t,y)dy+u(t,x)f(t,x,u),\quad x\in\bar D_1.
\end{eqnarray*}
Since $u_0|_{D_0} \le u_0$
%(this holds for the Dirichlet type boundary condition since $u_0$ is zero on boundary of $D_0$)
the inequality follows from Proposition \ref{comparison-prop}.
\end{proof}

For given $r>0$ and $x_0\in\RR^N$, let
$$
B_r(x_0)=\{x\in\RR^N\,|\, |x-x_0|<r\}.
$$

\begin{proposition}\label{new-prop1}
Let $ 0<\delta_0<1$ and $r_0>0$ be given positive numbers. Suppose that {\bf (H1)} holds. Then for any given positive integer $k$, there exist a positive number $\mu=\mu({r_0,\delta_0,k})$  and a positive integer $i=i({r_0,\delta_0,k})$ such that
\begin{equation}\label{Eq::1}
    \inf_{x\in B_{kr_0}(0)}\sum_{j=0}^{i}\frac{(\mathcal{K}^ju)(x)}{j!}\geq \mu \ \ \forall\  u\in L^{\infty}(\R^n), \ u\ge 0, \ \text{with}\ \int_{B_{r_0}(0)}u\,dx\ge \delta_0,
\end{equation}
where $\mathcal{K}u=\kappa*u$. In particular
$$
(e^{\mathcal{K}}u)(x)\geq \mu \quad \forall\ x\in B_{kr_0}(0).
$$
\end{proposition}

\begin{proof}
From ${\bf(H1)}$, we know that $\kappa$ is continuous and $\kappa(0) > 0$ so we can find $0<r <\frac{r_0}{2}$ such that $\kappa(x) \ge \frac{1}{2}\kappa(0)$ for every $x$ in $\bar{B}_{2r}(0)$. Now let $u\in L^{\infty}(\RR^N)$ be a nonnegative function satisfying $\int_{B_r(0)}u\,dx\ge \delta_0 .$
We claim that
\begin{equation}\label{Eq0}
    \inf_{x\in \bar{B}_{(m+1)r}(0)\setminus B_{mr}(0)}(\mathcal{K}^{m+1}u)(x)\geq \frac{[\delta_0\kappa(0)]^{m+1}}{2^{m+1}}\Pi_{i=1}^{m}\Big|B_{r}(ir{e_1})\cap B_{r}((i-1)re_1) \Big|\quad \forall\ m\ge 1
\end{equation}
where $e_1$ is the unit vector $(1,0,\dots,0) \in \RR^N.$

Observe from the definition of $r$ that
\begin{eqnarray*}
(\mathcal{K}u)(x) \geq \int_{B_{r}(0)} \kappa(y - x)u(y)dy\ge \frac{1}{2}\kappa(0) \int_{B_{r}(0)}u(y)dy
\ge \frac{1}{2}\kappa(0)\delta_0\quad \forall\ x\in \bar{B}_r(0).
\end{eqnarray*}
Hence
 \begin{equation}\label{Eq1}
 \inf_{x \in \bar{B}_{r}(0)}(\mathcal{K}u)(x) \;\; \ge \;\; \frac{1}{2}\kappa(0)\delta_0.
 \end{equation}
We proceed by induction to show that \eqref{Eq0} holds.

 To this end, let us first show that the claim holds for $m=1$. Observe  that for every $r\leq |x|\le 2r$ and $y\in B_{r}(\frac{rx}{|x|})$,
$|y-x|\leq |y-\frac{rx}{|x|}|+|x-\frac{rx}{|x|}|$ $ = |y-\frac{rx}{|x|}| + |x| - r$  $ < 2r$.  Hence, by \eqref{Eq1} for every $x\in \bar{B}_{2r}(0)\setminus B_{r}(0)$, we have
$$
\aligned
\mathcal{K}^2u(x)\geq \int_{B_{r}(\frac{rx}{|x|})}\kappa(y-x)(\mathcal{K}u)(y)dy\geq  \frac{1}{2}\kappa(0)\int_{B_r(\frac{rx}{|x|})}(\mathcal{K}u)(y)dy\ge \frac{\kappa(0)^2}{2^2}\delta_0\Big|B_{r}\big(\frac{rx}{|x|}\big)\cap B_r(0)\Big|.
\endaligned
$$
Since the Lebesgue measure is rotation invariant and $0<\delta_0<1$, we conclude from the last inequality that
\begin{equation}
    \inf_{x\in\bar{B}_{2r}(0)\setminus B_r(0)}\mathcal{K}^2u(x)\geq \frac{\kappa(0)^2}{2^2}\delta_0\Big|B_{r}(re_1)\cap B_r(0)\Big|\geq \frac{[\delta_0\kappa(0)]^2}{2^2}\Big|B_{r}(re_1)\cap B_r(0)\Big|
\end{equation}
 which proves \eqref{Eq0} for $m=1$.

Next, suppose that \eqref{Eq0} holds for some $m\ge 1$, we show that it also holds for $m+1$. Indeed, as in the previous case, observe that, as shown in the schematic below, we have the following:
$$
|y-x|\leq \big|y-(m+1)r\frac{ x}{|x|}\big|+\big|x-(m+1)r\frac{ x}{|x|}\big|< 2r \,\, $$
$\text{for }\ \ (m+1)r\le |x|\leq (m+2)r \ \text{and}\ y\in B_{r}(\frac{(m+1)rx}{|x|}).
$

\centerline{\includegraphics[width =2.8in]{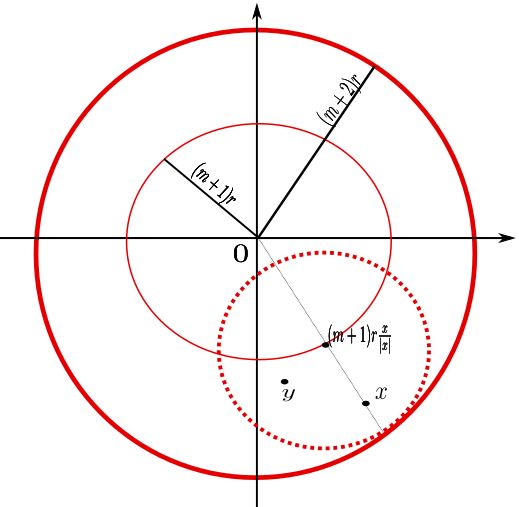}}
Observe also that
$$
B_{r}(\frac{(m+1)rx}{|x|})\cap\Big(\overline{B}_{(m+1)r}(0)\setminus B_{rm}(0)\Big)= B_{r}(\frac{(m+1)rx}{|x|})\cap\overline{B}_{(m+1)r}(0) \quad \forall x\ne 0.
$$  For notational convenience, let $B_{mr}(0) := B_0^{m}$ and $B_{(m+1)r}(0) := B_0^{m+1}.$ 
Using the induction hypothesis and recalling the choice of $r$, we obtain for every $x\in \bar{B}_{(m+2)r}(0)\setminus B_{(m+1)r}(0)$ that
$$
\aligned
\mathcal{K}^{m+2}u(x)
\geq & \int_{B_r(\frac{(m+1)rx}{|x|})}\kappa(y-x)\mathcal{K}^{m+1}u(y)dy\\
\ge & \frac{\kappa(0)}{2}\delta_0\int_{B_r(\frac{(m+1)rx}{|x|})}\mathcal{K}^{m+1}u(y)dy\\
\ge &  \frac{\kappa(0)}{2}\delta_0\Big|B_{r}(\frac{(m+1)rx}{|x|})\cap\big(\bar{B}_0^{m+1}\setminus B_0^{m}\big) \Big| \inf_{ x \in B_{r}(\frac{(m+1)rx}{|x|})\cap\big(\bar{B}_0^{m+1}\setminus B_0^{m}\big)}\mathcal{K}^{m+1}u(x)\\
=&  \frac{\kappa(0)}{2}\delta_0\Big|B_{r}(\frac{(m+1)rx}{|x|})\cap\bar{B}_0^{m+1} \Big| \inf_{ x \in B_{r}(\frac{(m+1)rx}{|x|})\cap\Big(\Bar{B}_0^{m+1}\setminus B_0^{m}\Big)}\mathcal{K}^{m+1}u(x) \\
\ge & \frac{[\delta_0\kappa(0)]^{m+2}}{2^{m+2}}\Big|B_{r}(\frac{(m+1)rx}{|x|})\cap\Bar{B}_0^{m+1} \Big|\Pi_{i=1}^{m}\Big|B_{r}(ir{e_1})\cap B_{r}((i-1)re_1) \Big|.
\endaligned
$$
Again, since the Lebesgue measure is rotation invariant, then $\Big|B_{r}(\frac{(m+1)rx}{|x|})\cap\overline{B}_{(m+1)r}(0) \Big|=\Big|B_{r}((m+1)re_1)\cap\overline{B}_{(m+1)r}(0) \Big|$, which together with the last inequality show that the claim also holds for $m+1$.

We then deduce that the claim holds for every $m\ge 1$. Now, by choosing $m\gg 1$ such that $B_{kr_0}(0)\subset B_{mr}(0)$, we can derive from \eqref{Eq0} that \eqref{Eq::1} holds with $i=m$.
\end{proof}

\subsection{Part metric}

In this subsection, we recall the decreasing property of part metric between two positive solutions of \eqref{main-nonlinear-eq}.

For given $u,v\in X^{++}$, the part metric between $u$ and $v$, denoted by $\rho(u,v)$,  is defined by
$$
\rho(u,v) = \inf\{\ln \alpha\,|\, \frac{1}{\alpha}u \le v \le \alpha u, \; \alpha \ge 1 \}.$$

\begin{proposition}
\label{new-prop2}
\begin{itemize}
\item[(1)]
For any $u_1,u_2\in X^{++}$ and $t>s$,
$\rho(u(t,\cdot;s,u_1),u(t,\cdot;s,u_2))\le \rho(u_1,u_2)$.

\item[(2)] For any $\delta > 0, \; \sigma > 0, \; M > 0$ and $\tau > 0$ with $\delta < M$ and $\sigma \le \ln{\frac{M}{\delta}},$ there is $\tilde{\sigma} > 0$ such that for any $u_0, \; v_0 \in X^{++}$ with $\delta \le u_0(x) \le M. \;\; \delta \le v_0(x) \le M$ for $x \in \RR^N$ and $\rho(u_0, v_0) \ge \sigma,$ there holds
$$
\rho(u(s+\tau,\cdot;s,u_0), u(s+\tau,\cdot ;s, v_0)) \le \rho(u_0, v_0) - \tilde{\sigma}\quad \forall\, s\in\RR.
 $$
\end{itemize}
\end{proposition}

\begin{proof}
(1)  See \cite[Proposition 5.1]{RaSh}.

(2)  See \cite[Proposition 3.4] {Kong}.
\end{proof}

\section{Uniqueness, stability, and frequency module of almost periodic solutions}

 In this section, we study the uniqueness, almost periodicity, and stability of a strictly positive   bounded  entire solution of \eqref{main-nonlinear-eq} and prove Theorem \ref{uniqueness-thm}.

We first prove two lemmas.

\begin{lemma}
\label{new-new-new-lm1}
Suppose that $g(t,x)$ is a uniformly continuous,  bounded  function in $t\in\RR$ and $x\in\bar D$, with $g(t,x)>0$ for all $(t,x) \in \RR \times D,$  and $f(t,x,u)$ satisfies {\bf (H2)}.
Then for any fixed $x\in\bar D$, the ODE
\begin{equation}
\label{new-new-new-ode1}
u_t=g(t,x)+uf(t,x,u)
\end{equation}
has at most one strictly positive   bounded  entire solution $u^*(t)$.
\end{lemma}

\begin{proof} It can be proved by properly modifying the arguments in \cite[Theorem 2.1]{Nka}. For completeness, we provide a proof in the following.

Fix $x\in\bar D$.
Suppose that \eqref{new-new-new-ode1} has two strictly positive   bounded   entire solutions $u_1^*(t)$ and $u_2^*(t)$, $u_1^*(t)\not = u_2^*(t)$. Without loss of generality, we may assume that $u_1^*(0)<u_2^*(0)$.
Then by comparison principle for ODEs,
$$
u_1^*(t)<u_2^*(t)\le M\quad \forall \, t \in\RR.
$$
By {\bf (H2)}, there is $\alpha>0$ such that
\begin{align}
\label{new-new-new-eq2}
\frac{d}{dt}\ln \big(\frac{u_1^*(t)}{u_2^*(t)}\big)&=
\frac{u^*_{1t}}{u^*_1}-\frac{u^*_{2t}}{u^*_2}\nonumber\\
&=\frac{g(t,x)}{u^*_1(t)}-\frac{g(t,x)}{u_2^*(t)}+f(t,x,u_1^*(t))-f(t,xu_2^*(t))\nonumber\\
&>f(t,x,u_1^*(t))-f(t,x,u_2^*(t))\nonumber \\
&\ge \alpha (u_2^*(t)-u_1^*(t))\quad \forall\, t\in\RR.
\end{align}
This implies that $\ln \big(\frac{u_1^*(t)}{u_2^*(t)}\big)$  increases  in $\RR$ and then there is some $0<c<1$ such that
$$
\frac{u_1^*(t)}{u_2^*(t)}\le \frac{u_1^*(0)}{u_2^*(0)}\le c<1\quad \forall\, t\le 0.
$$
Hence
$$
u_2^*(t)-u_1^*(t)=u_2^*(t)\Big(1-\frac{u_1^*(t)}{u_2^*(t)}\Big)\ge (1-c)u_2^*(t)\quad \forall \, t\le 0.
$$
This together with \eqref{new-new-new-eq2} implies that there is $\beta>0$ such that
$$
\frac{d}{dt}\ln \big(\frac{u_1^*(t)}{u_2^*(t)}\big)\ge \beta\quad \forall t\le 0.
$$
{Integrating the above inequality from $t$ to $0$ for $t\le 0$, we have}
$$
\ln \big(\frac{u_1^*(t)}{u_2^*(t)}\big)\le \ln\big(\frac{u_1^*(0)}{u_2^*(0)}\big)+\beta t\quad \forall\, t\le 0
$$
and then
$$
\frac{u_1^*(t)}{u_2^*(t)}\le \frac{u_1^*(0)}{u_2^*(0)} e^{\beta t}\quad \forall\, t\le 0.
$$
Letting $t\to -\infty$, we obtain
$$
\lim_{t\to - \infty}\frac{u_1^*(t)}{u_2^*(t)}=0,
$$
which contradicts $u_1^*(t)$ and $u_2^*(t)$ being \bk two strictly positive entire   bounded  solutions of \eqref{new-new-new-ode1}. Hence \eqref{new-new-new-eq1} has at most one strictly positive   bounded   entire solution.
\end{proof}

\begin{lemma}
\label{new-new-new-lm2} Suppose that $u^*(t,x)$ is a strictly positive and bounded measurable function on $\RR\times \bar D$, is differentiable in $t$ for each $x\in\bar D$, and satisfies \eqref{main-nonlinear-eq} for $t\in\RR$ and $x\in\bar D$, that is,
\begin{equation}
\label{new-aux-eq1}
\frac{\p u^*}{\p t}(t,x)=\int_D \kappa(y-x)u^*(t,y)dy+u^*(t,x)f(t,x,u^*(t,x)),\quad t\in\RR,\,\, x\in \bar D.
\end{equation}
Then  $u^*(t, x)$ is uniformly continuous in $t\in\RR$ and $x\in\bar D$, and   $\RR\ni t\mapsto u^*(t,\cdot)\in X$ is differentiable
and hence  $u^*(t,x)$ is a strictly positive bounded solution of \eqref{main-nonlinear-eq}.
\end{lemma} 

\begin{proof}
 We first  show that $u^*(t,x)$ is uniformly continuous in $t$ uniformly with respect to $x\in\bar D$ and is uniformly continuous in $x$ uniformly with respect to $t\in\RR$, i.e., for any $\epsilon>0$, there is $\delta>0$ such that for any $t_1,t_2\in\RR$ and
$x_1,x_2\in\bar D$ with $|t_1-t_2|<\delta$ and $|x_1-x_2|<\delta$,  there hold
$$
|u^*(t_1,x)-u^*(t_2,x)|<\epsilon \quad \forall\, x\in\bar D
$$
and
$$
|u^*(t,x_1)-u^*(t,x_2)|<\epsilon\quad \forall\, t\in\RR.
$$

Observe that $u_t^*(t,x)$ is a bounded function of $t\in\RR$ and $x\in\bar D$. This implies that $u^*(t,x)$ is uniformly continuous in $t$ uniformly with respect to $x\in\bar D$ and that
\begin{equation}
\label{new-new-new-eq4}
g(t,x):=\int_D \kappa(y-x)u^*(t,y)dy
\end{equation}
is uniformly continuous in $t\in\RR$ and $x\in\bar D$.

 Assume that $u^*(t,x)$ is not uniformly continuous in $x\in\bar D$ uniformly with respect to $t\in\RR$. Then there is $\epsilon_0>0$, $t_n\in\RR$, and $x_n,\bar x_n\in\bar D$ such that
$$
|x_n-\bar x_n|\le\frac{1}{n} \quad \forall\, n\ge 1,
$$
and
\begin{equation}
\label{new-new-new-eq5}
|u^*(t_n,x_n)-u^*(t_n,\bar x_n)|\ge \epsilon_0\quad \forall\, n\ge 1.
\end{equation}
Let $u_n(t)=u^*(t+t_n,x_n)$ and $\bar{u}_n(t)=u^*(t+t_n,\bar{x}_n),$  then
\begin{equation}
\label{u-n-eq}
\frac{du_n(t)}{dt}=g(t+t_n,x_n)+u_n(t)f(t+t_n, x_n, u_n)
\end{equation}
and
\begin{equation}
\label{bar-u-n-eq}
\frac{d\bar{u}_n(t)}{dt}=g(t+t_n,\bar x_n)+\bar{u}_n(t)f(t+t_n, \bar x_n, \bar u_n).
\end{equation}
Note that $u_n(t)$ and $\bar u_n(t)$ are uniformly continuous in $t\in\RR$.  Since  $u^*(t,x)$ is strictly positive and bounded,   there are $\delta_1 > 0, \; M \gg 1$ such  that
$$
 \delta_1 \le u^*(t,x) \le M \quad \forall \; t\in \R, x \in D.
$$
This yields that $u_n(t)$ and $\bar u_n(t)$ are uniformly bounded. Furthermore,  By {\bf (H2)} and the uniform continuity of $g(t,x)$,   we see that their derivatives are bounded, hence $u_n(t)$ and $\bar u_n(t)$ are equicontinuous. Therefore, using the usual diagonal argument and Arzela-Ascoli's theorem,  without loss of generality, we may assume that there are $u_1^*(t)$, $u_2^*(t)$, $g^*(t,x)$ and $f^*(t,x,u)$ such that
\begin{equation}
\label{u-n-limit-eq}
\lim_{n\to\infty}u_n(t)=u_1^*(t),\quad \lim_{n\to\infty}\bar u_n(t)=u_2^*(t),
\end{equation}
\begin{equation}
\label{g-n-limit-eq}
\lim_{n\to\infty} g(t+t_n,x+x_n)=g^*(t,x),\quad \lim_{n\to\infty} g(t+t_n,x+\bar x_n)=g^*(t,x),
\end{equation}
and
\begin{equation}
\label{f-n-limit-eq}
\lim_{n\to\infty} f(t+t_n,x + x_n,u)=f^*(t,x,u),\quad \lim_{n\to\infty} f(t+t_n,x+\bar x_n,u)=f^*(t,x,u)
\end{equation}
locally uniformly in $t\in\RR$, $x\in\bar D$, and $u\in\RR$.  By \eqref{u-n-eq}-\eqref{f-n-limit-eq}, $\frac{d u_n(t)}{dt}$ and $\frac{d \bar u_n(t)}{dt}$ also converge locally uniformly in $t\in\RR$ as $n\to\infty$. It then follows that $u_1^*(t)$ and $u_2^*(t)$ are differentiable in $t$ and   are two strictly positive bounded  entire solutions of
$$
u_t=g^*(t,0)+u f^*(t,0,u).
$$
By Lemma \ref{new-new-new-lm1}, $u_1^*(t)\equiv u_2^*(t)$, in particular,
$u_1^*(0)=u_2^*(0)$, which contradicts \eqref{new-new-new-eq5}.
Hence $u^*(t,x)$ is uniformly continuous in $t\in\RR$ and $x\in\bar D$.

Next, we prove that $\RR\ni t\mapsto u^*(t,\cdot)\in X$  is differentiable. By the uniform continuity of $u^*(t,x)$ in $t\in\RR$ and $x\in\bar D$,
$\RR\ni t\mapsto u^*(t,\cdot) \in  X$ is continuous. By \eqref{new-aux-eq1},  for each $x\in\bar D$, $u^*(\cdot,x)\in W_{\rm loc}^{1,1}(\RR)$. Hence $u^*(t,x)$ is both super-solution and sub-solution of \eqref{main-nonlinear-eq} on any interval $(a,b)$.  Then,
by Proposition \ref{comparison-prop},  for any given $t_0\in\RR$,
$$
u^*(t,\cdot)=u(t,\cdot;t_0,u^*(t_0,\cdot))\quad \forall \, t\ge t_0.
$$
This implies that $\RR\ni t\mapsto  u^*(t,\cdot)\in X$  is differentiable, and $u^*(t,x)$ is a strictly positive bounded entire solution of \eqref{main-nonlinear-eq}.
\end{proof}

Next, we prove Theorem \ref{uniqueness-thm}.

\begin{proof} [Proof of Theorem \ref{uniqueness-thm}]

(a)
%\rd First, Observe that using $(H2)$ we can see that $M \gg 1$ is a supersolution. Thus given any positive initial $u_0$, we can find $M \gg 1$ such that $u_0 \le M$. Hence, every positive solution to \eqref{new-new-new-ode1} is bounded. \bk
Suppose that  there are two strictly positive    bounded  entire solutions $u_1^*$ and $u_2^*$ of \eqref{main-nonlinear-eq}. If $u_1^* \neq u_2^*$,  then we can find $t_0 \in \RR$ such that $u_1^*(t_0, \cdot ) \neq u_2^*(t_0, \cdot ).$ This implies that  there is $\sigma > 0$ such that $\rho(u_1^*(t_0,\cdot), u_2^*(t_0,\cdot) \ge \sigma.$  By Proposition \ref{new-prop2}(1),
$$
\rho(u_1^*(t,\cdot),u_2^*(t,\cdot))\ge \sigma\quad \forall \, t\le t_0.
$$
Then by Proposition \ref{new-prop2}(2), there is $\tilde\sigma>0$ such that
\begin{equation}
\label{new-new-new-eq1}
\rho(u_1^*(t_0,\cdot),u_2^*(t_0,\cdot))\le \rho(u_1^*(t_0-k,\cdot),u_2^*(t_0-k,\cdot))-k\tilde\sigma\quad \forall\, k=1,2,\cdots.
\end{equation}
Note that $\rho(u_1^*(t_0-k,\cdot),u_2^*(t_0-k,\cdot))$ is bounded for $k\in\NN$.  This together with \eqref{new-new-new-eq1} implies that
$$
\rho(u_1^*(t_0,\cdot),u_2^*(t_0,\cdot))\le \rho(u_1^*(t_0-k,\cdot),u_2^*(t_0-k,\cdot))-k\tilde\sigma\to -\infty
$$
as $k\to\infty$,
which is a contradiction. Therefore, a strictly positive   bounded  entire solution of \eqref{main-nonlinear-eq} is unique.

\smallskip

\smallskip

(b) Suppose that $u^*(t,x)$ is a strictly positive bounded entire solution of \eqref{main-nonlinear-eq}.
We show that $u^*(t,x)$ is almost periodic in $t$ uniformly with respect to $x\in\bar D$. By Lemma \ref{new-new-new-lm2}, $u^*(t,x)$ is uniformly continuous in $t\in\RR$ and $x\in\bar D$. It then suffices to prove that for each $x\in\bar D$,
$u^*(t,x)$ is almost periodic in $t$. To this end, let $\{t_n\}$ and $\{s_n\}$ be any two sequences of $\RR$. By {\bf (H2)} and the uniform continuity of $u^*(t,x)$, without loss of generality, we may assume that there are $\bar f(t,x,u)$,
 $\tilde f(t,x,u)$, $\hat f(t,x,u)$  satisfying {\bf (H2)}, and $\bar u^*(t,x)$, $\tilde u^*(t,x)$, $\hat u^*(t,x)$ such that
$$
\lim_{n\to\infty}  f(t+t_n,x,u)=\bar f(t,x,u), \,\, \lim_{m\to\infty}\bar f(t+s_m,x,u)=\tilde f(t,x,u),\,\, \lim_{n\to\infty} f(t+t_n+s_n,x,u)=\hat f(t,x,u)
$$
locally uniformly in $(t,x,u)\in\RR\times \bar D\times \RR$,
and
$$
\lim_{n\to\infty} u^*(t+t_n,x)=\bar u^*(t,x),\,\, \lim_{m\to\infty} \bar u^*(t+s_m,x)=\tilde u^*(t,x),\,\, \lim_{n\to\infty} u^*(t+t_n+s_n,x)=\hat u^*(t,x)
$$
locally uniformly in $(t,x)\in\RR\times \bar D$. {Moreover, using \eqref{main-nonlinear-eq},  $\p_t u^*(t+t_n,x)$ also converges locally uniformly in $(t,x)\in\RR\times\bar D$ as $n\to\infty$, and then  $\bar u^*(t,x)$  is differentiable in $t$ and satisfies \eqref{main-nonlinear-eq} with $f$ being replaced by $\bar f$ for each $t\in\RR$ and $x\in\bar D$. By  Lemma \ref{new-new-new-lm2},   $\bar u^*(t,x)$ is a strictly positive bounded entire solution of \eqref{main-nonlinear-eq} with $f$ being replaced by $\bar f$. Similarly,
 $\tilde u^*(t,x)$ (resp.  $\hat u^*(t,x)$) is a strictly positive  {bounded} entire solution of \eqref{main-nonlinear-eq} with $f$ being replaced by $\tilde f$ (resp. $\hat f$).}
By Lemma \ref{ptw def}, $\tilde f(t,x,u)=\hat f(t,x,u)$. Then by (a), $\tilde u^*(t,x)=\hat u^*(t,x)$. By Lemma \ref{ptw def} again, $u^*(t,x)$ is almost periodic in $t$.

By the arguments similar to the proof of almost periodicity of $u^*(t,x)$ in $t$, we have that $u^*(t,x)$ is almost periodic in $x$ when $D=\RR^N$.

\smallskip

(c)   Suppose that $u^*(t,x)$ is a strictly positive bounded entire solution of \eqref{main-nonlinear-eq}. We prove that $u^*(t,x)$ is asymptotically stable with respect to strictly positive perturbation.
  First note that there are $\delta_1 > 0, \; M \gg 1$ such that
\begin{equation}\label{sol bound}
 \delta_1 \le u^*(t,x) \le M \quad \forall \; t\in \R, x \in D.
\end{equation}
For given  $u_0 \in X^{++}$ and $t_0 \in \RR$, let  $u(t,x;t_0,u_0) $ be the solution to \eqref{main-nonlinear-eq} with $u(t_0, x;t_0,u_0) = u_0(x)$.
Observe that, for some $0 < b \ll 1,\;bu^*(t,x)$ is a subsolution of \eqref{main-nonlinear-eq}, {and  $u\equiv M$ is a  supersolution  of \eqref{main-nonlinear-eq}  when  $M \gg 1$.  Therefore, we can find $0 < b \ll 1$ and $M \gg 1$ such that
$$bu^*(t_0,x) \le u_0(x) \le M \quad \forall \; x \in \bar D.$$
By Proposition \ref{comparison-prop},
\begin{equation}\label{bound eqn}
bu^*(t,x) \le u(t,x;t_0,u_0) \le M \quad  \forall \; t \ge t_0, \; x\in\bar  D.
\end{equation}}

 Let $\rho(t;t_0)=\rho(u(t+t_0,\cdot;u_0),u^*(t+t_0,\cdot))$  for every $t\ge 0$. We claim that
\begin{equation}\label{zz-5}
    \limsup_{t\to\infty}\sup_{t_0\in\mathbb{R}}\rho(t;t_0)=0.
\end{equation}
Suppose on the contrary that \eqref{zz-5} is false. Then we can find  sequences $\{t_{0,n}\}_{n\ge 1}$ and $\{t_n\}_{n\ge 1}$ with $t_n\ge1+ n$ for each $n\geq1$  such that
\begin{equation*}
    \sigma_0:=\inf_{n\ge 1}\rho(t_n;t_{0,n})>0.
\end{equation*}
By proposition \ref{new-prop2}(1), we know that $\rho(t;t_{0,n})\ge \rho(t_n;t_{0,n})\geq \sigma_0$ for every $n\ge 1$ and $0\le t\le t_n$. Thus, by  \eqref{sol bound}, \eqref{bound eqn} and  proposition \ref{new-prop2}(2), there is $\tilde{\delta}>0$ such that
$$
\rho(t+1;t_{0,n})\leq\rho(t;t_{0,n})-\tilde{\delta}\quad \forall\ n\geq 1,\ 0\le t<t_n.
$$
In particular, since $n<t_n$ for each $n\ge 1$,
$$
\rho(n+1;t_{0,n} )\leq \rho(n;t_{0,n})-\tilde{\delta}\le \cdots\leq \rho(0;t_{0,n})-(n+1)\tilde{\delta} \quad \forall\ n\ge 1.
$$
Hence we have
\begin{equation}\label{zz-4}
0<\sigma_0\leq \rho(t_n;t_{0,n})\leq \rho(n+1;t_{0,n})\leq \rho(0;t_{0,n})-(n+1)\tilde{\delta}\quad \forall\  n\ge 1.
\end{equation}
This yields a contradiction since  $\rho(0;t_{0,n})=\rho(u^*(t_{0,n},\cdot),u_0)\le \ln(\frac{M}{\delta}\big)$ for all $n\ge 1$. Hence we conclude that \eqref{zz-5} must hold. Now, \eqref{zz-5} implies that
$$
\lim_{t\to\infty}\sup_{t_0\in\mathbb{R}}\|u^*(t+t_0,\cdot)-u(t+t_0,\cdot;t_0,u_0)\|_{\infty}=0.
$$
This establishes the asymptotic stability of $u^*(t,x)$ with respect to  strictly positive perturbations.

\smallskip

(d)  Suppose that $u^*(t,x)$ is a strictly positive bounded entire solution of \eqref{main-nonlinear-eq}.  We prove that $\mathcal{M}(u^*)\subset\mathcal{M}(f)$.
  For any given sequence $\{t_n\}$ in $\RR$,  suppose that  $f(t +t_n, x,u) \to f(t,x,u)$ uniformly on bounded sets.
By   (a) and (b),   there  is a subsequence $\{t_{n_k}\}$  of $\{t_n\}$ such that, $u^*(t + t_{n_k}, x) \to u^*(t,x)$ uniformly on bounded sets, as $k \to \infty.$ Similarly, for any given sequence $\{x_n\}$ in $\RR^N$, if \bk $f(t,x+x_n,u)\to f(t,x,u)$ uniformly on bounded sets, then there is a subsequence $\{x_{n_k}\}$ of $\{x_n\}$ such that $u^*(t,x+x_{n_k})\to u^*(t,x)$ as $k\to\infty$ locally uniformly.
It then follows  from Proposition \ref{mod cont} that $\mathcal{M}(u^*)\subset\mathcal{M}(f)$.
\end{proof}

\section{Existence and nonexistence of positive   bounded  entire solutions}

In this section, we study the existence  of a strictly positive   bounded  entire solution of \eqref{main-nonlinear-eq} and prove Theorem \ref{existence-thm}.

\begin{proof} [Proof of Theorem \ref{existence-thm}]
 (a)
First, suppose that \eqref{main-nonlinear-eq} has a strictly positive   bounded  entire solution $u^*(t,x)$. {By {\bf (H2)}, $f_{\inf}(u):=\inf_{t\in\RR,x\in\bar D} f_u(t,x,u)$ is continuous in $u\ge 0$ and $f_{\inf}(u)<0$ for $u\ge 0$.  Let $u^*_{\inf}=\inf_{t\in\RR,x\in\bar D} u^*(t,x)$ and $u^*_{\sup}=\sup_{t\in\RR,x\in\bar D} u^*(t,x)$.
Then
 for any  $0<\lambda\le  - u_{\inf}^* \cdot \sup_{u\in [0, u_{\sup}^*]} f_{\inf}(u)$, we have
\begin{align}
\label{lambda-ineq}
f(t,x,u^*(t,x))-f(t,x,0)&=\int_0^ 1 \frac{d}{ds} f(t,x, s u^*(t,x)) ds\nonumber\\
&=u^*(t,x)\int_0^ 1 f_u(t,x,su^*(t,x))d s\nonumber\\
& \le -\lambda\quad\forall\, t\in\RR,\, x\in\bar D.
\end{align}}
This implies that
\begin{align*}
    u^*_t&= \int_D \kappa(y-x)u^*(t,y)dy + u^*f(t,x,u^*(t,x))\\
&=   \int_{D}\kappa(y-x)u^*(t,y)dy+  u^* (f(t,x,0) +f(t,x,u^*(t,x))-f(t,x,0))\\
&\le \int_D \kappa(y-x) u^*(t,y)dy+u^*(f(t,x,0)-\lambda)\quad \forall\, t\in\RR,\,x\in\bar D.
\end{align*}
It then follows that $ \lambda_{PE}(a)\ge \lambda>0$, where $a(t,x)=f(t,x,0)$.

Next, suppose that $\lambda_{PE}(a) > 0$. Let $M \gg 1.$ Then $u(t,x) \equiv M$ is a supersolution of (\ref{main-nonlinear-eq}).
 By Proposition \ref{comparison-prop}, $u(t, \cdot ; -K, M) \leq M$.  This implies that $u(t, x ; -K, M)$ decreases as $K$ increases. Hence we can define
\begin{equation}
\label{u-star-def-eq}
(0\le) u^*(t,x):=\lim_{K\to\infty} u(t,x;-K,M){(\le M)}\quad \forall\,\, t\in\mathbb{R},\,\, x\in \bar D.
\end{equation}
It is clear that $u^*(t,x)$ is measurable in $(t,x)\in\RR\times\bar D$. Moreover, note that
$$
u_t(t,x;-K,M)=\int_D\kappa(y-x)u(t,y;-K,M)dy+u(t,x;-K,M) f(t,x,u(t,x;-K,M))
$$
for all $ t>-K$ and $ x\in\bar D$.
This together with the Dominated Convergence Theorem implies that, for each fixed $x\in\bar D$,
\begin{equation}
\label{proof-eq00}
 u^*_t(t,x)= \int_D \kappa(y-x)u^*(t,y)dy + u^*f(t,x,u^*(t,x))\quad \forall\, t\in\mathbb{R},
\end{equation}
and then $u^*(\cdot,x)\in W_{\rm loc}^{1,1}(\RR)$.

In the following, we prove that $u^*(t,x)$ is strictly positive.  We do so in two steps.

\smallskip

\noindent {\bf Step 1.} In this step, we prove that  {\it   there is $r_x>0$ such that
 \begin{equation}
\label{proof-eq0}
\inf_{t\in\mathbb{R},y\in B_{r_x}(x)\cap D}u^*(t,y)>0.
\end{equation}
}

 Let $\lambda \in \Lambda_{PE}(a)$ be such that $0< \lambda < \lambda_{PE},$ $\lambda_{PE} - \lambda \ll 1$. Let
$\phi\in \mathcal{X}^+$ satisfy $\inf_{t\in\RR}\phi(t,x)\ge \not \equiv 0$, $\|\phi\|_{\mathcal{X}}=1$,
{ for each $x\in\bar D$,  $ \phi \in W^{1,1}_{\rm loc}(\RR)$ }  and
$$
\lambda \phi (t,x)\leq L\phi(t,x)\quad {\rm for}\quad a.e.\,  t\in\RR.
$$
By {\bf (H2)},  $\left( f(t,x,0) - f(t,x,b\phi) - \lambda \right)\phi(t,x) \leq 0$ for  $0 < b \ll 1$. Thus for each $x\in\bar D$, 
 $u(t,x) = b\phi(t,x)$ solves
\begin{eqnarray*}\label{sub}
\frac{\partial u(t,x)}{\partial t} &\leq& \int_D \kappa(y - x)u(t,y)dy  + a(t,x)u(t,x) - \lambda u(t,x)\\
&=&  \int_D \kappa(y - x)u(t,y)dy  + u(t,x)f(t,x,u) \\
& & + \left( f(t,x,0) - f(t,x,u) - \lambda \right)u(t,x)\\
&\leq&  \int_D \kappa(y - x)u(t,y)dy  + u(t,x)f(t,x,u)\quad {\rm for}\,\, a.e. \, t\in\RR.
\end{eqnarray*}
 Hence, $b\phi$ is a subsolution of (\ref{main-nonlinear-eq}). Therefore, by Proposition \ref{comparison-prop},
\begin{equation}
\label{new-eq0}
u(t,x;-K,M)\ge u(t, x ; -K, b\phi(-K,x)) \geq b\phi(t,x)\quad \forall\, t\ge -K,\,\, x\in \bar D.
\end{equation}

Since $ \inf_{t\in\RR} \phi(t,x)\ge \not\equiv 0,$ we can find $x_0 \in D$ such that
 $$ \delta_1 := \inf_{t\in\RR} b\phi(t,x_0) > 0.
$$
 Moreover, by the continuity of $\inf_{t\in\RR} \phi(t,x)$ in $x$, we have
\begin{equation}
\label{new-eq1}
     \inf_{t\in\RR} b\phi(t,x) \ge  \delta_1/2 \text \; {for} \;x \in D_0:=B_{r_0}(x_0)\cap D\;\; \text{for some} \;\; r_0 > 0.
\end{equation}
Observe  that there is $m >0$ such that $\|f(t,x,u(t,x;-K,M))\| \le m$ for all $t\ge -K$ and $x\in D$. Thus  $u(t,x;-K,M)$ solves
$$\p_t u \ge\int_{D}
\kappa(y-x)u(t,y)dy - mu(t,x)\quad \forall\, t>-K,\,\, x\in \bar D.$$
This together with \eqref{new-eq0} implies that
\begin{equation}
\label{new-eq2}
u(t+1 ,x;-K,M)\ge e^{-m}\big( e^{\mathcal{K}} b \phi(t,\cdot) \big)(x)\quad \forall\, t\ge -K,\, x\in D,
\end{equation}
where
 $\mathcal{K}(u)(x) =\int_{D}
\kappa(y-x)u(y)dy$ for $u\in X$. Hence
\begin{equation}
\label{new-eq2-1}
u^*(t,x)\ge   e^{-m}\big( e^{\mathcal{K}} b \phi(t,\cdot) \big)(x)\quad \forall\, t\in \mathbb{R},\, x\in D.
\end{equation}
By the arguments of  Proposition \ref{new-prop1} and \eqref{new-eq1}, for each $x\in\bar D$, there are $r_x>0$ and  $\mu_x>0$ such that
$$
\big(e^{\mathcal{K}}b\phi(t,\cdot)\big)(y)\ge \mu_x\quad \forall\, t\in\mathbb{R},\,\, y\in B_{r_x}(x)\cap D.
$$
This together with \eqref{new-eq2-1} implies
\eqref{proof-eq0}.

\medskip

\noindent {\bf Step 2.} In this step, we prove that
\begin{equation}
\label{proof-eq000}
\inf_{t\in\mathbb{R},x\in\bar D} u^*(t,x)>0.
\end{equation}

 In the case that $D$ is bounded, \eqref{proof-eq000} follows from \eqref{proof-eq0}.

In the case that $D=\RR^N$, by the almost periodicity of
$a(t,x)$  in $x$, for  any  given $\varepsilon > 0,$ there is  $r_{\varepsilon} > 0$ such that any ball of radius $r_{\varepsilon}$ contains some $\Tilde{x} \in T_{\varepsilon}$,  where
$$
T_{\varepsilon} := \{\Tilde{x} : |a(t,x) - a(t,x + \Tilde{x})| < \varepsilon \; \forall \; (t,x) \in \RR \times \RR^N\}.
 $$
For given $\varepsilon > 0,$ we can find a  sequence $\{\Tilde{x}_n\}_{n \in \NN} \in T_{\varepsilon}$ such that
\begin{equation}\label{union}
\RR^N = \underset{n\in \NN}{\bigcup}{\color{red} B_{2r_{\varepsilon}}}(\Tilde{x}_n),
\end{equation}
where
  $B_{2r_{\varepsilon}}(\Tilde{x}_n):= \{ x\in \RR^N : \|x - \Tilde{x}_n \| <2 r_{\varepsilon} \} $.
  %where each point in $\RR^N$ intersects only a finite number of balls in the union.
Let $\varepsilon = \frac{\lambda}{2}$.  Then
\begin{eqnarray*}\label{subb}
\frac{\partial (\phi(t,x))}{\partial t} &\leq& \int_{\RR^N} \kappa(y - x)\phi(t,y)dy  + a(t,x + \Tilde{x}_n)\phi(t,x) + (a(t,x) - a(t,x+\Tilde{x}_n) - \lambda) \phi(t,x)\\
&\le&  \int_{\RR^N} \kappa(y - x)\phi(t,y)dy  + a(t,x + \Tilde{x}_n)\phi(t,x) + (\varepsilon - \lambda) \phi(t,x)\\
&=&   \int_{\RR^N} \kappa(y - x)\phi(t,y)dy  + a(t,x + \Tilde{x}_n)\phi(t,x)  - \frac{\lambda}{2} \phi(t,x).
\end{eqnarray*}
Hence, for some $0 < \Tilde{b} <1, \;\; \Tilde{b}\phi$ is a subsolution of
\begin{equation}
\label{dirichlet-kpp-eq translate}
 \p_t u(t,x) = \int_{\RR^N}
\kappa(y-x)u(t,y)dy+u(t,x)f(t,x+\Tilde{x}_n,u(t,x)),\quad x\in\RR^N.
\end{equation}
 By Proposition \ref{comparison-prop}, we have
 $$
 \Tilde{b}\phi(t,x) \le u(t,x+\Tilde{x}_n;-K,M )\quad  {\rm for} \,\,t\ge -K,\,\,  x \in \RR^N, \; \Tilde{x}_n \in T_{\varepsilon}.
$$
By arguments similar to \eqref{new-eq2}, we have
\begin{equation}
\label{new-eq3}
u(t+1,x+\tilde x_n;-K,M)\ge e^{-m}e^{\mathcal{K}} \tilde b \phi(t,\cdot),\quad \forall t\ge -K,\,\, x\in\RR^N.
\end{equation}
Without loss of generality, we may assume $x_0=0$ in \eqref{new-eq1}. Then
by   Proposition \ref{new-prop1}, \eqref{new-eq1}, and  \eqref{new-eq2},  there is $\tilde \delta_2>0$ such that
$$
u(t+1,x+\tilde x_n;-K,M)\ge \tilde \delta_2\quad \forall\,\, t\ge -K,\,\, x\in { B_{2r_\epsilon}(0)}.
$$
This  together with \eqref{union} implies that
\begin{equation}
\label{new-eq4}
u(t+1,x;-K,M)\ge \tilde \delta_2\quad \forall\, t\ge -K, \,\, x\in\RR^N,
\end{equation}
which implies \eqref{proof-eq000}.

By \eqref{u-star-def-eq}, \eqref{proof-eq00}, \eqref{proof-eq000}, and Lemma \ref{new-new-new-lm2}, $u^*(t,x)$ is a strictly positive bounded entire solution of \eqref{main-nonlinear-eq}.

\smallskip

(b) Assume that $\lambda_{PL}<0$. For any $u_0\ge 0$,
$$
u(t,x;0,u_0)\le \Phi(t,0)u_0\quad \forall\, t\ge 0, \,\, x\in D.
$$
Note that
$$
\limsup_{t\to\infty}\frac{\ln \|\Phi(t,0)u_0\|}{t}\le \lambda_{PL}<0.
$$

Hence
$$
0\le \limsup_{t\to\infty} \|u(t,\cdot;0,u_0)\|\le \limsup_{t\to\infty}\|\Phi(t,0)u_0\|=0.
$$
The theorem thus follows.
\end{proof}

{
\begin{remark}
\label{positive-entire-solution-rk}
As mentioned in Remark \ref{remark-0}, the definitions of $\lambda_{PL}(a),\lambda_{PL}^{'}(a),\lambda_{PE}(a),$ and $\lambda_{PE}^{'}(a)$ apply to general $a(t,x)$ which is bounded and uniformly continuous.
 When $f(t,x,u)$ is not assumed to be almost periodic in $t,$  if $\lambda_{PE}(a)>0$, then $u^*(t,x)$ defined in \eqref{u-star-def-eq} is bounded on $\RR\times\bar D$,  is  differentiable in $t$ and $\inf_{t\in\mathbb{R}} u^*(t,x)>0$ for each $x\in\bar D$, and satisfies \eqref{main-nonlinear-eq} for each $t\in\RR$ and $x\in\bar D$. Hence $\p_t u^*(t,x)$ is bounded on $\RR\times\bar D$. We can also prove that $u^*(t,x)$ is continuous in $x\in\bar D$.  In fact, let $g^*(t,x)=\int_{D}\kappa(y-x)u^*(t,y)dy$.  It is clear  that $g^*(t,x)$ is uniformly continuous in $t\in\mathbb{R}$ and $x\in\bar D$.
 for any $x_0\in\bar D$ and $\{x_n\}\subset \bar D$ with
$x_n\to x_0$, without loss of generality, we may assume that $u^*(t,x_n)\to \tilde u^*(t)$,
$g(t,x_n)\to g(t,x_0)$, and $f(t,x_n,u^*(t,x_n))\to f(t,x_0,\tilde u^*(t))$  as $n\to\infty$ locally uniformly in $t\in\mathbb{R}$. By \eqref{proof-eq00}, we have
$$
\tilde u^*_t=g(t,x_0)+\tilde u^*(t) f(t,x_0,\tilde u^*(t))\quad \forall\,\, t\in\mathbb{R}
$$
and
$$
u^*_t(t,x_0)=g(t,x_0)+u^*(t,x_0) f(t,x_0, u^*(t,x_0))\quad \forall\,\, t\in\mathbb{R}.
$$
By \eqref{proof-eq0} and Lemma \ref{new-new-new-lm1}, $\tilde u^*(t)= u^*(t,x_0)$.  It then follows that $u^*(t,x)$ is also continuous in $x\in\bar D$.
But $u^*(t,x)$ may not be strictly positive. However,
 if $D$ is bounded, then $u^*(t,x)$ is a
 strictly positive entire solution of \eqref{main-nonlinear-eq} and  is asymptotically stable with respect to positive perturbations.
\end{remark} }

\section{Monotonicity of $\lambda_{PE}(a,D)$ in $D$}

In this section, we investigate the monotonicity of $\lambda_{PE}(a,D)$ in $D$ and prove Theorem \ref{pev-thm1}.

\begin{proof} [Proof of Theorem \ref{pev-thm1}]
Let $D_1\subset D_2$ be given. Without loss of generality, we may assume that $\lambda_{PE}(a,D_2)=0$. For otherwise,
we can replace $a(t,x)$ by $a(t,x)-\lambda_{PE}(a,D_2)$. It then suffices to prove that $\lambda_{PE}(a,D_1)\le 0$.
We prove it by contradiction.

{ First, assume that $\lambda_{PE}(a,D_1)>0$. Let $\delta>0$ be such that
$\lambda_{PE}(a-\delta,D_1)>0$.   By Theorem \ref{existence-thm}, there is a strictly positive bounded entire solution $u_1^*(t,x)$ of
 \begin{equation}
\label{new-new-eq1-0}
u_t=\int_{D_1}\kappa(y-x)u(t,y)dy+u(t,x)(a(t,x)-\delta -u(t,x)),\quad x\in \bar D_1.
\end{equation}
For given $M>0$, let  $u_2(t,x;-K,M)$ be the solution of
\begin{equation}
\label{new-new-eq1}
u_t=\int_{D_2}\kappa(y-x)u(t,y)dy+u(t,x)(a(t,x)-\delta/2 -u(t,x)),\quad x\in \bar D_2
\end{equation}
with $u_2(-K,x;-K,M)=M$.   By Propositions \ref{comparison-prop} and  \ref{new-comparision-prop},
\begin{equation}
\label{monotone-eq1}
u_1^*(t,x)\le  u_2(t,x;-K,M)\quad \forall\, t\ge -K,\,\, x\in  \bar D_1,\quad M\gg 1,
\end{equation}
and
\begin{equation}
\label{monotone-eq2}
u_2(t,x;-K,M) \le M\quad \forall\, t\ge -K,\,\, x\in\bar D_2,\quad M\gg 1.
\end{equation}
Fix $M\gg 1 $.
By  the arguments of  Theorem \ref{existence-thm},
\begin{equation}
\label{monotone-eq5}
u_2^*(t,x):=\lim_{K\to\infty} u_2(t,x;-K,M)(\le M),\quad t\in\RR,\,\, x\in \bar D_2
\end{equation}
is well defined,  and satisfies \eqref{new-new-eq1} for all $t\in\RR$ and $x\in\bar D_2$.

Next, we  claim that $u_2^*(t,x)$ is strictly positive.  We divide the proof of the claim  into two cases.

\smallskip

\noindent {\bf Case 1.} $D_2$ is bounded. Note that there is $m>0$ such that
$$
a(t,x)-\delta/2-u_2(t,x;-K,M)\ge -m\quad \forall\, t\ge -K,\,\, x\in\bar D_2.
$$
This together with \eqref{new-new-eq1} and Proposition \ref{comparison-prop} implies that
$$
u_2(t,\cdot;-K,M)\ge e^{-m} e^{\mathcal{K}_2} u_2(t-1,\cdot;-K,M)\quad \forall\, t\ge -K+1,
$$
where $\mathcal{K}_2 u=\int_{D_2}\kappa(y-x)u(y)dy$ for $u\in C_{\rm unif}^b(\bar D_2)$.
By \eqref{monotone-eq1}, there is $\delta_0>0$ such that
$$
\int_{D_2} u_2(t-1,x;-K,M)dx\ge \delta_0 \quad \forall\,\,  t\ge -K+1,\quad x\in \bar D_2.
$$
This together with the arguments of Proposition \ref{new-prop1} implies that there is $\tilde \delta_0>0$ such that
$$
u_2(t,x;-K,M)\ge \tilde\delta_0\quad \forall\, t\ge -K+1,\quad x\in \bar D_2.
$$
It then follows that
$$
u_2^*(t,x)\ge \tilde \delta_0\quad \forall\, t\in\RR,\,\, x\in\bar D_2.
$$
Hence the claim holds in the case that $D_2$ is bounded.

\smallskip

\noindent {\bf Case 2.} $D_2=\RR^N$. By the almost periodicity of $a(t,x)$ in $x$, there are $\{x_n\}\subset\RR^N$ and
$r>0$ such that
$$
\RR^N=\cup_{n=1}^\infty B_r(x_n),
$$
and
$$
|a(t,x+x_n)-a(t,x)|\le \delta/2\quad \forall\, t\in\RR,\quad x\in\RR^N.
$$
Then
\begin{align*}
\p_t u_{1}^*(t,x)&=\int_{D_{1}}\kappa(y-x)u_{1}^*(t,y)dy + u_{1}^*(t,x)(a(t,x)-\delta-u_{1}^*(t,x)\\
&\le \int_{D_{1}}\kappa(y-x)u_{1}^*(t,y)dy +u_{1}^*(t,x)(a(t,x+x_n)-\delta/2-u_{1}^*(t,x))\quad \forall\, t\in\RR,\,\, x\in
\bar D_{1}.
\end{align*}
This together with  Propositions \ref{comparison-prop} and  \ref{new-comparision-prop}  implies that
$$
u_2(t,x+x_n;-K,M)\ge u_{1}^*(t,x) \quad \forall\, t\ge -K,\,\, x\in\bar D_{1}
$$
and then
$$
u_2^*(t,x)\ge u_1^*(t,x-x_n)\quad \forall\, t\in\RR,\,\, x-x_n\in\bar D_{1}.
$$
By the arguments in Case 1, there is $\tilde \delta_0>0$ such that
$$
u_2^*(t,x)\ge \tilde \delta_0\quad \forall\, t\in\RR,\,\, x\in B_r(x_n),\,\, n\ge 1.
$$
Therefore, $u^*(t,x)$ is strictly positive and the claim also holds in the case $D_2=\RR^N$.

\smallskip

Now, by  Lemma \ref{new-new-new-lm2}, $u_2^*(t,x)$ is  uniformly continuous in $t\in\RR$ and $x\in \bar D_2$.
Hence $u_2^*(t,x)$ can be used as a test function in the definition of $\Lambda_{PE}(a,D_2)$.
\begin{equation}
\label{new-new-eq1-000}
-\frac{\p u_2^*}{\p t}+\int_{D_2}\kappa(y-x)u_2^*(t,y)dy+a(t,x) u_2^*(t,x)\ge \frac{\delta}{2} u_2^*(t,x),\quad t\in\RR,\,\, x\in \bar D_2.
\end{equation}
This
implies  that $\lambda_{PE}(a,D_2)\ge \frac{\delta}{2} >0$, which is a contradiction. Hence
$\lambda_{PE}(a,D_1)\le 0$.
The theorem is thus proved.}
\end{proof}

\bk


\begin{thebibliography}{99}
%\bibitem{Bochner}
%S. Bochner, A new approach to almost periodicity,
%{\it  Proc. Nat. Acad. Sci.}, {\bf  48} (1962), 2039-2043.
\bibitem{BaLi0} X. Bai and F.  Li,
 Optimization of species survival for logistic models with non-local dispersal,
  {\it Nonlinear Anal. Real World Appl.}, {\bf 21} (2015), 53-62.


\bibitem{BaZh}
P. Bates and G. Zhao, Existence, Uniqueness and Stability of the stationary solution to a nonlocal evolution equation arising in population dispersal,
{\it J. Math. Anal. Appl.}, \textbf{332}(9) (2007),  428-440.


\bibitem{BaLi} X. Bao and W.-T. Li,  Propagation phenomena for partially degenerate nonlocal dispersal models in time and space periodic habitats,
    {\it  Nonlinear Anal. Real World Appl.}, {\bf 51} (2020), 102975, 26 pp.


\bibitem{BeCoVo1}
H. Berestycki, J.  Coville, and H. Vo,  Persistence criteria for populations with non-local
ispersion, {\it J. Math. Biol.},  {\bf 72} (2016), 1693-1745.

% \bibitem{CoElRo1} C. Cortazar, M. Elgueta, and J.D. Rossi, Nonlocal diffusion problems that approximate the heat equation with Dirichlet
% boundary conditions, {\it Israel J. Math.}, {\bf  170} (2009), 53-60.

% \bibitem{CoElRoWo} C. Cortazar, M. Elgueta, J.D. Rossi, and
%  N. Wolanski, How to approximate the heat equation with Neumann boundary
% conditions by nonlocal diffusion problems, {\it Arch. Ration. Mech. Anal.}, {\bf  187} (1) (2008), 137-156.

\bibitem{Co1}
J. Coville,
 On a simple criterion for the existence of a principal eigenfunction of some nonlocal operators,
  {\it Journal of Differential Equations}, \textbf{249} (2010), 2921-2953.

%\bibitem{Co2}
%J. Coville,  On Uniqueness and Monotonicity of Solutions of Nonlocal reaction diffusion equation,
% {\it Annali di Matematica},  \textbf{185}(3) (2006), 461-485.

\bibitem{Co3} J. Coville, Nonlocal refuge model with a partial control,
{\it Discrete Contin. Dyn. Syst.}, {\bf  35} (4) (2015),
1421-1446.

%\bibitem{CoDu}
%J. Coville and L.  Dupaigne,
 %Propagation Speed of travelling fronts in nonlocal reaction-diffusion equations,
% {\it  Nonlinear Anal.},  \textbf{60} (2005), 797-819.

\bibitem{CoDaMa}
J. Coville, J.  D\'avila, and S. Mart\'inez,  Existence and Uniqueness of solutions to a nonlocal equation with monostable nonlinearity, {\it  SIAM J. Math. Anal.},  \textbf{39} (2008), 1683-1709.

 \bibitem{CoDaMa1} J. Coville, J. Davila, and  S. Martinez, Pulsating fronts for nonlocal dispersion and KPP nonlinearity,
 {\it Ann. Inst. H. Poincaré Anal. Non Linéaire}, {\bf 30} (2013), 179-223.

\bibitem{Eva} L. C. Evans, Partial Differential Equations, Graduate Studies in Mathematics, Vol. 19, American Mathematical Society (2002).

\bibitem{DeShZh} P. De Leenheer, W. Shen, and A. Zhang,
 Persistence and extinction of nonlocal dispersal evolution equations in moving habitats,
  {\it Nonlinear Anal. Real World Appl.}, {\bf 54} (2020), 103110, 33 pp.

 \bibitem{Fif}P.C. Fife, An integrodifferential analog of semilinear parabolic PDEs, in: Partial Differential Equations and Applications, in: Lecture Notes in Pure and Appl. Math., vol. 177, Dekker, New York,
1996, 137-145.


\bibitem{fink}
A.M. Fink,
 Almost Periodic Differential Equations,  Lecturen Notes in Mathematics, No 377, Springer-Verlag, New York, 1974.

%  \bibitem{GaRo}
% J. Garcia-Melian and J.D. Rossi,
%  On the principal eigenvalue of some nonlocal diffusion problems,
%  {\it  J. Differ. Equ.},  \textbf{246} (2009), 21-38.

 \bibitem{GaRo1} J. Garcia-Melian, J.D. Rossi, A logistic equation with refuge and nonlocal diffusion,
 {\it Commun. Pure
Appl. Anal.}, {\bf  8} (6) (2009), 2037-2053.

 \bibitem{GrHiHuMiVi}
 M. Grinfeld, G.  Hines, V. Hutson, K. Mischaikow, and G.T. Vickers,
  Non-local Dispersal,
  {\it  Differ, Integr. Equ.},  \textbf{18} (2005), 1299-1320.



%  \bibitem{Hes} P.  Hess,
% Periodic-parabolic boundary value problems and positivity,
% Pitman Research Notes in Mathematics Series, {\bf 247},  Longman, New York, 1991.

%   \bibitem{HeShZh}
% G. Hetzer, W.  Shen, and A.  Zhang,
%  Effects of Spatial Variations and Dispersal Strategies on Principal eigenvalues of dispersal operators and spreading speeds of monostable equations,
%  {\it  Rocky Mount. J. Math.},  \textbf{43}(2) (2013), 1147-1175.

% \bibitem{HuSh}
%J. Huang and W. Shen,
%Speeds of spread and propagation of KPP models in time almost and space periodic media,
%{\it SIAM J. Appl. Dyn. Syst.}, {\bf 8} (2009), no. 3, 790-821.



% \bibitem{HuGr}
%V. Hutson and M.  Grinfeld, Non-local dispersal and bistability,
%{\it  Euro. J. Appl. math.},  \textbf{17} (2006), 221-232.


% \bibitem{HuPo} J. Huska and P.  Polacik,
%  The principal Floquet bundle and exponential separation for linear parabolic equations,
%   {\it J. Dynam. Differential Equations}, {\bf 16} (2004), no. 2, 347-375.

% \bibitem{HuPoSa} J.  Huska, P. Polacik,  and M.V. Safonov,  Harnack inequalities, exponential separation, and perturbations of principal Floquet bundles for linear parabolic equations,
%      {\it Ann. Inst. H. Poincar\'e Anal. Non Lin\'eaire}, {\bf 24} (2007), no. 5, 711-739.

\bibitem{HuMaMiVi}
V. Hutson, S.  Martinez,  K.  Mischaikow, and G.T. Vickers,
 The evolution of Dispersal, {\it J. Math. Biol.},  \textbf{47} (2003),  483-517.

% \bibitem{HuShVi} V. Hutson,  W. Shen, and G.T.  Vickers,
%   Spectral theory for nonlocal dispersal with periodic or almost-periodic time dependence,
% {\it Rocky Mountain J. Math.},  {\bf 38} (2008), no. 4, 1147-117.


\bibitem{KaLoSh}
C.-Y. Kao, Y. Lou, and W.  Shen,
 Random Dispersal vs Non-Local Dispersal,
 {\it  Discr. Cont. Dyn. Syst.},  \textbf{26}(2) (2010), 551-596.

\bibitem{Kong}
Kong, Lang;
{\it Spatial Spread Dynamics of Monostable Equations in Spatially Locally Inhomogeneous Media with Temporal Periodicity},  Doctoral Dissertation (Auburn University) (2013)

%  \bibitem{LiCoWa} F. Li, J. Coville, and X. Wang,
%   On eigenvalue problems arising from nonlocal diffusion models,
%   {\it  Discrete Contin. Dyn. Syst.}, {\bf  37} (2017), no. 2, 879-903.

\bibitem{LiSuWa}
W.-T. Li, Y.-J.  Sun, and Z.-C. Wang,
 Entire Solutions in the Fisher-KPP equation with Nonlocal Dispersal,
  {\it Nonlinear Analysis: Real World Appl.},  \textbf{11}(4) (2010), 2302-2313.

  \bibitem{LiWaZh} W.-T. Li, J.-B. Wang, and X.-Q. Zhao,
  Spatial dynamics of a nonlocal dispersal population model in a shifting environment,
   {\it J. Nonlinear Sci.}, {\bf  28} (2018), no. 4, 1189-1219.

 \bibitem{LiZh} X.  Liang and T. Zhou,
  Spreading speeds of nonlocal KPP equations in almost periodic media,
  {\it  J. Funct. Anal.}, {\bf 279} (2020), no. 9, 108723, 58 pp.

  \bibitem{LuPaLe} F. Lutscher, E. Pachepsky,  and M.A. Lewis, The effect of dispersal patterns on stream populations,
{\it SIAM Rev.}, {\bf 47} (4) (2005), 749-772.

\bibitem{Nka} M. N. Nkashama,
Dynamics of logistic equations with non-autonomous bounded coefficients,
{\it
Electron. J. Differential Equations}, {\bf  2000}, No. 02, 8 pp.


% \bibitem{NadRos}
% G. Nadin and L.  Rossi,
%  Propagation phenomena for time heterogeneous KPP reaction-diffusion equations,
%  {\it  J. Math. Pures Appl.},  {\bf (9) 98} (2012), no. 6, 633-653.

\bibitem{OnSh}
M.A. Onyido and W. Shen,
Nonlocal dispersal equations with almost periodic dependence. I. Principal spectral theory,  {\it Journal of Differential equations}, {\bf 295} (2021), 1-38.

\bibitem{OnSh1} M.A. Onyido and W.  Shen,  Corrigendum to: "Nonlocal dispersal equations with almost periodic dependence. I. Principal spectral theory'' [J. Differ. Equ. 295 (2021) 1–38], {\it J. Differential Equations} {\bf  300} (2021), 513-518.



\bibitem{Paz}
A.L. Pazy,  Semigroups of Linear Operators and Applications to Partial Differential Equations. Springer, New York, 1983.

\bibitem{RaSh}
N. Rawal and W.  Shen,
 Criteria for the existence and lower bounds of principal eigenvalues of time periodic nonlocal dispersal operators and applications,
  {\it J. Dynam. Differential Equations}, {\bf  24} (2012), 927-954.


\bibitem{RaShZh}
N. Rawal, W.  Shen, and A.  Zhang,
 Spreading speeds and traveling waves of nonlocal monostable equations in time and space periodic habitats,
  {\it Discrete Contin. Dyn. Syst.},  {\bf 35} (2015), no. 4, 1609-1640.

\bibitem{Sh} W. Shen,
Stability of transition waves and positive entire solutions of Fisher-KPP equations with time and space dependence, {\it Nonlinearity} {\bf 30(6)} (2017) , 3466-3491.

% \bibitem{ShXi} W. Shen and X.  Xie,  Spectral theory for nonlocal dispersal operators with time periodic indefinite weight functions and applications,
%     {\it  Discrete Contin. Dyn. Syst. Ser. B}, {\bf  22} (2017), no. 3, 1023-1047.

% \bibitem{ShXi1}
% W. Shen and X.   Xie,
%  On principal spectrum points/principal eigenvalues of nonlocal dispersal operators and applications,
%   {\it Discrete and Continuous Dynamical Systems}, \textbf{35} (2015), 1665-1696.

\bibitem{ShXi2}
W. Shen and X.  Xie,
Approximations of random dispersal operators/equations by nonlocal dispersal operators/equations,
 {\it J. Differential Equations}, {\bf  259} (2015), no. 12, 7375-7405.

% \bibitem{ShYi} W. Shen and Y.  Yi,
% Almost automorphic and almost periodic dynamics in skew-product semiflows,  Mem. Amer. Math. Soc.,  {\bf  136} (1998), no. 647.



\bibitem{ShZh1}
W. Shen and A. Zhang,
Spreading speeds for monostable equations with nonlocal dispersal in space periodic habitats,
{\it Journal of Differential Equations}, \textbf{249} (2010), 747-795.

 \bibitem{ShZh2}
W.  Shen and A. Zhang,
 Stationary solutions and spreading speeds of nonlocal monostable equations in space periodic habitats, {\it  Proc. Amer. Math. Soc.}, {\bf 140} (2012), no. 5, 1681-1696.

\bibitem{ShVo}
Z. Shen and H.-H. Vo,
 Nonlocal dispersal equations in time-periodic media: principal spectral theory, limiting properties and long-time dynamics,
  {\it J. Differential Equations}, {\bf  267} (2019), no. 2, 1423-1466.

   \bibitem{SuLiLoYa} Y.-H. Su, W.-T. Li, Y. Lou, and F.-Y. Yang,
   The generalised principal eigenvalue of time-periodic nonlocal dispersal operators and applications,
    {\it J. Differential Equations}, {\bf 269} (2020), no. 6, 4960-4997.

 \bibitem{Tur}  P. Turchin, Quantitative Analysis of Movement: Measuring and Modeling Population Redistribution
in Animals and Plants, Sinauer Associates, 1998.


\bibitem{ZhZh1} G.-B. Zhang and X.-Q. Zhao,
  Propagation dynamics of a nonlocal dispersal Fisher-KPP equation in a time-periodic shifting habitat,
   {\it J. Differential Equations}, {\bf 268} (2020), no. 6, 2852-2885.

 \bibitem{ZhZh2} G.-B.  Zhang and X.-Q. Zhao,  Propagation phenomena for a two-species Lotka-Volterra strong competition system with nonlocal dispersal,
     {\it  Calc. Var. Partial Differential Equations}, {\bf  59} (2020), no. 1, Paper No. 10, 34 pp.

\end{thebibliography}
\end{document}